\documentclass[smallextended]{svjour3}

\smartqed
\usepackage{amssymb,amsxtra}
\usepackage{dsfont,latexsym}
\usepackage{url}

\usepackage{mathrsfs}
\renewcommand{\mathcal}{\mathscr}

\usepackage{color}
\definecolor{citation}{rgb}{0.2,0.6,0.2}
\definecolor{formula}{rgb}{0.1,0.2,0.6}
\definecolor{url}{rgb}{0,0,0.4}
\usepackage[colorlinks=true,linkcolor=formula,urlcolor=url,citecolor=citation]{hyperref}

\newcommand{\CC}{\mathcal{C}}
\newcommand{\GG}{{\mathcal G}}
\newcommand{\MM}{{\mathcal M}}
\newcommand{\FF}{\mathcal{F}}
\newcommand{\TT}{\mathcal{T}}
\newcommand{\KK}{\mathcal{K}}
\newcommand{\XX}{\mathcal{X}}

\newcommand{\R}{{\mathds R}}
\newcommand{\N}{{\mathds N}}

\newcommand{\dys}{\displaystyle}
\newcommand{\eps}{\varepsilon}
\newcommand{\aand}{\,{\mbox{ and }}\,}
\newcommand{\st}{\,{\mbox{ s.t. }}\,}

\newlength{\defbaselineskip}
\begin{document}
           
\title
{Local and global minimizers for a
variational energy  involving a fractional norm\thanks{The first author has been supported by Istituto Nazionale di Alta Matematica ``F.~\mbox{Severi}'' (Indam) 
 and by ERC grant 207573 ``Vectorial problems''.
The second author has been partially supported by National Science Foundation (NSF) grant 0701037.
The third author has been partially supported by FIRB ``Project Analysis and Beyond''.
The second and the third authors have been supported by ERC grant ``$\epsilon$ Elliptic Pde's and Symmetry of Interfaces and Layers for Odd Nonlinearities''.
}
}

\author{Giampiero Palatucci \and Ovidiu Savin \and Enrico Valdinoci
}

\institute{G. Palatucci \at
Dipartimento di Matematica \hfill Dipartimento di Matematica\\
Universit\`a degli Studi di Parma \hfill Universit\`a di Roma ``Tor Vergata''\\
Parco Area delle Scienze, 53/A \hfill Via della Ricerca Scientifica, 1\\
43124 Parma, Italy \hfill 00133 Roma, Italy \\
\email{giampiero.palatucci@unimes.fr}
\and O. Savin \at
Department of Mathematics,  
Columbia University \\
2990 Broadway \\
New York, NY 10027, USA\\
\email{savin@math.columbia.edu}
\and E. Valdinoci \at
Dipartimento di Matematica,  
Universit\`a degli Studi di Milano \\
Via Saldini, 50\\ 
20133 Milano, Italy\\
\email{enrico@math.utexas.edu}
}

\date{Received: date / Accepted: date}

\maketitle

\begin{abstract}
We study existence, uniqueness and other geometric
properties of the 
minimizers of the energy functional
$$
\|u\|^2_{H^s(\Omega)}+\int_\Omega W(u)\,dx,
$$
where $\|u\|_{H^s(\Omega)}$ denotes the total contribution from $\Omega$ in the $H^s$ norm of $u$ and $W$ is a double-well potential.
We also deal with the solutions of the related fractional elliptic Allen-Cahn equation on the entire space $\mathbb{R}^n$.

The results collected here will also be useful for forthcoming papers, where the second and the third author will study the $\Gamma$-convergence and the density estimates for level sets of minimizers.



\end{abstract}

\setcounter{tocdepth}{2}

\tableofcontents

\section{Introduction}

In this paper we study existence, uniqueness, some qualitative properties and related issues for the 
minimizers of a nonlocal energy functional involving a Gagliardo-type norm.

\vspace{1mm}

Let $\Omega\subseteq \R^{n}$ be an open domain and denote by $\CC\Omega$ its complement. We deal with the functional $\FF$ defined by
\begin{equation}\label{energia}
\dys \FF(u,\Omega)=\KK(u,\Omega)+\int_\Omega W(u)\,dx,
\end{equation}
where $\KK(u,\Omega)$ is given by
\begin{equation}\label{kinetic}
\KK(u,\Omega)=\frac{1}{2}\int_\Omega\!\int_\Omega \frac{|u(x)-u(y)|^2}{|x-y|^{n+2s}}\,dx\,dy+
\int_\Omega\!\int_{\CC\Omega} \frac{|u(x)-u(y)|^2}{|x-y|^{n+2s}}\,dx\,dy,
\end{equation}
with $s\in(0,1)$, and the function $W$ is a smooth double-well potential with wells at $+1$ and $-1$; i.e., $W$ is a non-negative function vanishing only at $\{-1,+1\}$.

The functional in \eqref{energia} is a non-scaled Allen-Cahn-Ginzburg-Landau-type energy with its kinetic term $\KK$ given by some nonlocal fractional integrals, in place of the classical Dirichlet integral. 
The energy $\KK(u,\Omega)$ of a function $u$, with prescribed boundary data outside $\Omega$,
can be view as the contribution in $\Omega$ of the 
(squared) $H^s$~(semi)norm of~$u$
$$
\int_{\R^n}\!\int_{\R^n}\frac{|u(x)-u(y)|^2}{|x-y|^{n+2s}}\,dx\, dy.
$$
Nonlocal models involving the $H^s$ norm are quite important in physics, since they naturally arise from many problems that exhibit long range interactions among particles.

In the specific case in (\ref{energia}) with the potential $W$ given by a double-well function, an adequate scaling of the kinetic term $\KK$ brings to the energy for a liquid-liquid two-phase transition model of nonlocal type. A $\Gamma$-convergence theory for such energy has been recently developed by two of the authors in~\cite{sv1}. They show that suitable scalings of the functional~$\FF$ $\Gamma$-converge to the standard minimal surface functional when $s\in[1/2,1)$ and to the nonlocal one when $s \in (0,1/2)$. As in the classical case with the singular perturbation given by the Dirichlet energy, the functional in \eqref{energia} is strictly related to the elliptic Allen-Cahn equation, which is of nonlocal character in this framework.

The nonlocal analogue of the Allen-Cahn equation is given by the 
following Euler-Lagrange equation for the energy $\FF(u):=\FF(u, \R^n)$
\begin{equation}\label{equazione}
\dys (-\Delta)^s u(x) + W'(u(x)) = 0 \ \ \text{for any} \ x \in \R^n,
\end{equation}
As usual, for any $s\in (0,1)$, $(-\Delta)^s$ denotes the $s$-power of the Laplacian operator and, omitting a multiplicative constant $c=c(n,s)$, we have
$$
\dys (-\Delta)^su(x)\, =\, P.V.\int_{\R^n}\frac{u(x)-u(y)}{|x-y|^{n+2s}}\,dy \, = \, \lim_{\eps\to 0} \int_{\CC B_\eps(x)}\frac{u(x)-u(y)}{|x-y|^{n+2s}}\,dy.
$$
Here $B_\eps(x)$ denotes the $n$-dimensional ball of radius~$\eps$, centered at~$x\in\R^n$ (and the standard notation $B_\eps=B_\eps(0)$ will be also used). ``$P.V.$'' is a commonly used abbreviation for ``in the principal value sense''. In the sequel, we will often omit the $P.V.$ notation
in front of the integrals, for simplicity of notation.

\vspace{2mm}

In the same spirit of a celebrate De Giorgi conjecture about the level sets of the solutions of the elliptic analogue of \eqref{equazione},
it seems natural 
to study the solutions $u$ of \eqref{equazione} that satisfy the 
following two conditions:
\begin{equation}\label{derivata}
\partial_{x_n}u(x)>0 \   \ \ \text{for any} \ x\in \R^n
\end{equation}
and, possibly,
\begin{equation}\label{limiti}
\lim_{x_n\rightarrow\pm\infty}u(x',x_n)=\pm1, \ \ \text{for any} \ x'\in\R^{n-1}.
\end{equation}
We refer to \cite{cabre,SV09,SV09b,FV11,cabre2010,cabre2011,cabre21th} 
for several results in this direction.
Here, by means of a technical variation of the classical sliding method, we can prove that the solutions
of the fractional elliptic Allen-Cahn
equation~(\ref{equazione}) that enjoy~\eqref{derivata}
and~\eqref{limiti} have also to satisfy a minimizing 
property for the 
functional~$\FF$ defined in \eqref{energia}, as stated here below:

\begin{theorem}\label{pro_min}
Let~$s\in (0,1)$ and let $u\in C^1(\R^n)$ be a solution of
\begin{equation}\label{eq56}
-(-\Delta)^s u(x)=W'(u(x)), \ \ \text{for any} \ x\in \R^n.
\end{equation}
 Suppose that
\begin{equation}\label{eq_strict}
\partial_{x_n} u(x)>0, \ \ \text{for any} \ x\in\R^n
\end{equation}
and
\begin{equation}\label{lim56}
\lim_{x_n\rightarrow\pm\infty}u(x',x_n)=\pm1, \ \ \text{for any} \ x'\in\R^{n-1}.
\end{equation}
Then,
for any~$r>0$, we have that
\begin{equation}\label{def_locmin}
\dys \FF(u,B_r)\le\FF(u+\phi,B_r) \ \ \text{for any measurable} \ \phi \ \text{supported in} \ B_r.
\end{equation}
\end{theorem}
In the literature, \eqref{def_locmin} is generally referred by saying that $u$ is a {\em local minimizer} for~$\FF$ in the domain~$B_r$.

The proof of Theorem~\ref{pro_min} follows a classical sliding argument
(see, e.g., Lemma~9.1
in~\cite{ssv} and also~\cite{alberti2001} and~ \cite{ambrosio} for a 
different 
variational approach for the classical local functional),
but we need here to operate some modifications due to the non-locality of 
the fractional operators $(-\Delta)^s$.

\vspace{2mm}

In the case of $\Omega$ being an open one-dimensional set, we will carefully characterize such class of minimizers, showing that they are monotone increasing and unique up to translations.
Moreover, by further regularity assumptions on the potential $W$, we have that the 1-D~minimizers satisfy certain regularity properties and then we will analyze their asymptotic behavior and the one of their derivative (see Theorem~\ref{theorem_unod} below).
Precisely, we denote by
\begin{equation}\label{def_x}
{\mathcal{X}} = \big\{ f\in L^1_{\textrm{loc}}(\R)\st
\lim_{x\rightarrow\pm\infty} f (x)=\pm 1\big\}
\end{equation}
the space of admissible functions and we suppose that the double-well potential $W$ belongs to $C^{2}(\R)$ and satisfies
\begin{equation}\label{CON6}
W''(\pm1)>0.
\end{equation}
Then, we prove the following theorem.

\begin{theorem}\label{theorem_unod}
Let $\FF$  given by~\eqref{energia}. Then there exists a unique (up to translations) nontrivial global minimizer $u^{(0)}\in\XX$ of the energy $\FF$ which is strictly increasing.
The minimizer $u^{(0)}$ solves the equation~\eqref{equazione} and is unique (up to translations) also in the class of monotone solutions of~\eqref{equazione}.
Moreover, $u^{(0)}$ belongs to~$C^2(\R)$ and there exists a constant~$C\ge 1$ such that
\begin{equation}\label{sss}
|{u^{(0)}} (x)-\,{\rm sign}\,(x)|\le C\,|x|^{-2s}
\ \text{and} \ \
\big|\big({u^{(0)}}\big)' (x)\big|\le C\,|x|^{-(1+2s)}
\end{equation}
for any large $x\in\R$.
\end{theorem}

\vspace{2mm}

As a further matter, exploiting Theorem~\ref{theorem_unod}, we will be able to construct a minimizer in higher dimensions $u^\ast$ and we will estimate the energy~\eqref{energia} of $u^\ast$ on the ball $B_R$, proving that, as $R$ gets larger and larger, the contribution in $\KK(u^\ast, B_R)$ from $\CC B_R$ becomes negligible if $s\geq 1/2$, however when $s<1/2$ this does not happen.

Precisely, we consider the functional $\mathcal{G}:\mathcal{X}\to\R\cup\{+\infty\}$ defined as follows
\begin{equation}\label{def_g}
{\mathcal{G}}(u) = \begin{cases}
\dys \liminf_{R\rightarrow +\infty}\, \frac{1}{R^{1-2s}}\FF(u,[-R,R]) & \text{if} \ s\in (0,\, 1/2),\\
\\
\displaystyle \liminf_{R\rightarrow+\infty}\, \frac{1}{\log R}\FF(u,[-R,R])
& {\mbox{ if $s=1/2$,}}\\
\\
\FF(u,\R) & {\mbox{ if $s\in(1/2,\,1)$,}}
\end{cases}
\end{equation}
where, for every $I\in\R$, $\FF(\cdot, I)$ is defined by (\ref{energia}). The functional $\mathcal{G}$ is given by the natural scaling of the energy $\FF$, in the sense that, for any $s\in (0,1)$, we have that $\GG(u^{(0)})$ is finite, where $u^{(0)}$ is the minimizer in~Theorem~\ref{theorem_unod}. In this respect, we say that the function $u^{(0)}$ is a {\em global minimizer} for $\FF$ if $\GG(u^{(0)})$ is finite and $u^{(0)}$ is a local minimizer for $\FF$ in any $B_r$ (see~Section~\ref{sec_minimi}).
Finally, it is worth mentioning that the scaling in~\eqref{def_g} also appears in the $\Gamma$-convergence analysis in~\cite{sv1} (see, in particular, ~Theorem 1.2 and Theorem 1.3 there).

\vspace{1mm}

We extend~$u^{(0)}$ to all the dimensions by setting,
for any~$x\in\R^n$ (and~$n\geq 2$),

\begin{equation}\label{defust}
u^* (x)=u^*(x_1,\dots,x_n):= u^{(0)}(\varpi x_n).
\end{equation}
where $\varpi$ is a constant needed just to keep track of the dependence
of~$(-\Delta)^s$ on the dimension, given by
\begin{equation}\label{varpi} \varpi:= \frac{1}{
\left(
\displaystyle\int_{\R^{n-1}} \displaystyle\frac{d\zeta}{
(1+|\zeta|^2)^{(n+2s)/2}}
\right)^{\!\frac{1}{2s}}
}.\end{equation}
This constant\footnote{
Of course, we could have kept track of the normalization constant~$\varpi$ in the definition of the fractional Laplacian operator (instead of in~\eqref{defust}), so that~\eqref{defust} reduces to the simpler~$u^*(x)=u^{(0)}(x_n)$. However, we preferred this choice both for consistency with~\cite{sv2011,sv0,sv1} and because most of the computations here are not complicated at all by this setting.
} also appears in~\cite{caffarelli} 
and~\cite{ambrosio2010}.

\vspace{2mm}

We prove the following theorem.
\begin{theorem}\label{theorem_0956}
Let $\GG$ be the 1-D functional defined by~\eqref{def_g} and let $u^{*}$ be defined by \eqref{defust}. Then, for any~$r>0$, we have that
\begin{equation}\label{8swqkfhfee}
\FF(u^*,B_r)\le\FF(u^*+\phi,B_r)
\end{equation}
for any measurable~$\phi$ supported in~$B_r$.

Also, the following results hold as $R\to+\infty$.
\begin{itemize}
\item[(i)]{If $s\in (0, 1/2)$, then
$$
c_1 \leq \frac{1}{R^{n-2s}}\int_{B_R}\int_{\CC B_R}\!\!\frac{|u^*(x)-u^*(y)|^2}{|x-y|^{n+2s}}dx\, dy \leq c_2.
$$
}\vspace{1mm}
\item[(ii)]{If $s=1/2$, then
$$
\frac{\FF(u^\ast,B_R)}{R^{n-1}\log R} \to b^\ast \ \ \text{and} \ \ \frac{1}{R^{n-1}\log R}\int_{B_R}\int_{\CC B_R}\!\!\frac{|u^*(x)-u^*(y)|^2}{|x-y|^{n+1}}dx\, dy \to 0.
$$
}\vspace{1mm}
\item[(iii)]{If $s\in (1/2, 1)$, then
$$
\frac{\FF(u^\ast,B_R)}{R^{n-1}} \to b^\ast \ \ \text{and} \ \ \frac{1}{R^{n-1}}\int_{B_R}\int_{\CC B_R}\!\!\frac{|u^*(x)-u^*(y)|^2}{|x-y|^{n+2s}}dx\, dy \to 0,
$$
}\vspace{1mm}
\end{itemize}\vspace{-3mm}
where $c_1$ and $c_2$ are positive constants and $\dys b^\ast=\frac{\omega_{n-1}}{\varpi}\GG(u^{(0)})$.
\vspace{2mm}

Moreover, there exists~$C>0$ such that for any~$R\ge2$
and~$\delta\in (0, 1/2)$ we have
\begin{equation}\label{1-del}
\FF (u^*,B_{R}\setminus B_{(1-\delta)R})\le C \delta R^{n-1}.
\end{equation}
\end{theorem}

\vspace{3mm}

Finally, it is worth noticing that, in order to prove all the above cited results, we need to perform careful computations on the strongly nonlocal form of the functional $\FF$. Hence, it was important for us to understand
some  
modifications of the classical techniques to deal with the fractional 
energy term, in particular to manage the contributions coming from far.
Therefore, in the Appendix we collect some general and independent results involving the Gagliardo-type norm in~(\ref{kinetic}), to be applied here and in~\cite{sv2011,sv0,sv1}, like construction of barriers, compactness results and various estimates, as well as regularity and limit properties for the solutions of equation~\eqref{equazione}.
\vspace{1mm}

We prove
Theorem~\ref{pro_min} in Section~\ref{78},
Theorem~\ref{theorem_unod} in Section~\ref{sec_minimi},
and Theorem~\ref{theorem_0956} in Section~\ref{sd97qq1}.
Some preliminary results on 
1-D minimizers on intervals are collected in Section~\ref{intermediate}.

\vspace{3mm}

\section{Minimization by sliding - Proof of 
Theorem~\ref{pro_min}}\label{78}

In this section, we prove the minimization result via sliding method stated in~Theorem~\ref{pro_min}. First, we need the following lemma, in which
we point out that the problem of
minimizing the energy in a given ball has a solution. 

\begin{lemma}\label{5656}
Let~$R>0$ and~$u_o:\R^n\rightarrow\R$ be a measurable function.
Suppose that there exists a measurable function~$\tilde u$ which
coincides with~$u_o$ in~$\CC B_R$ and such that~$\FF(\tilde 
u,B_R)<+\infty$. Then, there exists a measurable function~$u_\star$ such 
that~$\FF(u_\star,B_R)\le\FF(v,B_R)$ for any
measurable function~$v$ which
coincides with~$u_o$ in~$\CC B_R$.\end{lemma}

\begin{proof} We take a minimizing sequence, that is, let~$u_k$
be such that~$u_k=u_o$ in~$\CC B_R$, $\FF(u_k,B_R)\le
\FF(\tilde u,B_R)$ and
\begin{equation}\label{4543} \lim_{k\rightarrow+\infty} \FF(u_k,B_R)=
\inf \FF(v,B_R),\end{equation}where the infimum is taken over
any~$v$ that coincides with~$u_o$ in~$\CC B_r$.
Then,~\eqref{4543} and Lemma~\ref{C.B}
give that, up to subsequence, $u_k$ converges almost everywhere
to some~$u_\star$. Thus, the desired result follows from~\eqref{4543}
and Fatou Lemma.\qed
\end{proof}

\vspace{2mm}

Now, we are in position to prove that
every monotone solution of equation (\ref{equazione}),
satisfying  the limit condition (\ref{limiti}),
is a local minimizer for the corresponding energy functional $\FF$.

\begin{proof}[Proof of Theorem~\ref{pro_min}] We argue by contradiction. Suppose that
there exist~$r$, $c_o>0$ and~$\phi$ supported in~$B_r$
such that~$\FF(u,B_r)-\FF(u+\phi,B_r)\ge c_o$.

By choosing $\tilde{u}=u+\phi$ in Lemma~\ref{5656}, 
we have that $\FF(\tilde{u}) \le c_o + \FF(u,B_r)<+\infty$ and then
we can take~$u_\star$ minimizing~$\FF(v, B_r)$
among all the measurable functions~$v$ such that~$v=u$ in~$\CC B_r$.
Since we assumed by contradiction that $u$ is not a minimizer, then
there exists~$P\in\R^n$
such that
\begin{equation}\label{d77qi1msydfaJJSAS}
u(P)<u_\star(P).
\end{equation}

By cutting at the levels~$\pm 1$, which possibly makes~$\FF$ decrease,
we see that~$|u_\star|\le 1$.

Moreover, by the minimizing property of~$u_\star$,
\begin{equation}\label{5651}{\mbox{
$(-\Delta)^s u_\star(x)+W'(u_\star(x))=0$
for any~$x\in B_r$}}\end{equation}
and then, by Lemma~\ref{lem_co} in the Appendix and \cite[Theorem 3.3]{CS10}, $u_\star$ is continuous up to the boundary of $B_r$.\footnote{
We observe that the solutions of~\eqref{equazione} at every point are
also viscosity solutions (according to Definition~2.5
of~\cite{CS10}). Indeed, if~$v\ge u$ and~$v(x_o)=u(x_o)$, we have
that
\begin{eqnarray*}
&& \int_{\R^n}\frac{v(y)-v(x_o)}{|x_o-y|^{n+2s}}\,dy
\ge  \int_{\R^n}\frac{u(y)-v(x_o)}{|x_o-y|^{n+2s}}\,dy=
\int_{\R^n}\frac{u(y)-u(x_o)}{|x_o-y|^{n+2s}}\,dy.
\end{eqnarray*}
Note also that the solutions in the distributional sense satisfy the comparison principle and then they are solutions in the viscosity sense, too (see, e.g.,~\cite[Section~2.2]{silvestre} and, in particular, Proposition~2.2.6 there).
This allows us to use the results of~\cite{CS10}.
}\label{nota}

We claim that
\begin{equation}\label{90198}
|u_\star|<1.\end{equation}
To check this, let us argue by contradiction and suppose that, say,~$u_\star(\bar x)= +1$,
for some~$\bar x\in\R^n$.

Since~$|u|<1$ by our assumptions and~$u_\star=u$ in~$\CC B_r$, we have
that~$\bar x\in B_r$.
Then~\eqref{5651} and the fact that~$W'(+1)=0$ would give that
$$
\int_{\R^n}\frac{1-u_\star(y)}{|\bar x-y|^{n+2s}}\,dy \,=\,0.
$$
Since the integrand is always nonnegative, $u_\star$ must be identically
equal to~$+1$. But this is in contradiction with the fact 
that~$u_\star=u$ in~$\CC B_r$,
hence it proves~\eqref{90198}.\vspace{1mm}

Now, we note that there is a first contact point in $B_r$. Though this looks quite obvious, we give full details for the reader's convenience.

First, we claim that
there exists~$\bar k\in\R$ such that, 
\begin{equation}\label{56above}{\mbox{
if~$k\ge\bar k$, then~$u(x',x_n+k)\ge u_\star(x)$
for any~$x=(x',x_n)\in\R^n.$}}
\end{equation}
Again, this looks quite straightforward, but we give a complete argument:
we argue by contradiction
and we suppose that, for any~$k\in\N$, there exists~$x^{(k)}=
({x^{(k)}}',x^{(k)}_n)\in\R^n$ for which~$u
({x^{(k)}}',x^{(k)}_n+k)< u_\star({x^{(k)}})$.
Since~$u$ is monotone and~$k\ge0$, it follows that~$u
({x^{(k)}})< u_\star({x^{(k)}})$ and therefore~$x^{(k)}\in B_r$.

Thus, up to subsequence, we suppose that
$$ \lim_{k\rightarrow+\infty} x^{(k)}=x_\star,$$
for some~$x_\star$ in the closure of~$B_r$. Consequently, 
by~\eqref{lim56},
$$ +1\ =\  \lim_{k\rightarrow+\infty} u
({x^{(k)}}',x^{(k)}_n+k)\ \le \ 
\lim_{k\rightarrow+\infty} u_\star({x^{(k)}})\ =\ u_\star(x_\star)\ \le\ \sup_{B_r}
u_\star.$$\vspace{1mm}

Since this is in contradiction with~\eqref{90198},
we have proved~\eqref{56above}.
\vspace{1mm}
\noindent 
\\ Then, by~\eqref{56above} and the monotonicity of~$u$,
we have that, if~$k>\bar k$, then~$u(x',x_n+k)> u_\star(x)$
for any~$x=(x',x_n)\in\R^n$. We take~$\bar k$ as small as possible with
this property, i.e.,~$u(x',x_n+k)\ge u_\star(x)$ for any~$k\ge\bar k$
and any~$x\in\R^n$,
and there exist an infinitesimal sequence~$\eta_j>0$ and points~$p^{(j)}
\in\R^n$
for which~$u({p^{(j)}}',p^{(j)}_n+\bar k-\eta_j)\le u_\star(p^{(j)})$.
\vspace{1mm}

So, recalling~\eqref{d77qi1msydfaJJSAS}, we have that
$ u(P)\, <\, u_\star(P)\, \le \, u(P',P_n+\bar k)$ 
and then the monotonicity of~$u$ implies that
\begin{equation}\label{gsd78df1}
\bar k>0.
\end{equation}

We claim that
\begin{equation}\label{gsd78df2}
p^{(j)}\in B_r.
\end{equation}
Indeed, if~$p^{(j)}$ belonged to~$\CC B_r$ we would have that
$$ u({p^{(j)}}',p^{(j)}_n+\bar k-\eta_j)\ \le\  u_\star(p^{(j)})\ =\ 
u(p^{(j)}).$$
Hence, by the monotonicity of~$u$,
we would have that~$\bar k-\eta_j\le 0$ and so, by taking the limit
in~$j$, that~$\bar k\le0$. This is in contradiction with~\eqref{gsd78df1}
and so~\eqref{gsd78df2} is proved.
\vspace{1mm}

Then, by~\eqref{gsd78df2}, we may suppose that
$\dys \lim_{j\rightarrow+\infty} p^{(j)}=\zeta,$ 
for some~$\zeta$ in the closure of~$B_r$.
As a consequence,
the function~$w(x):=
u(x',x_n+\bar k) -u_\star(x)$ satisfies~$w(x)\ge0$
for any~$x\in\R^n$ and~$w(\zeta)=0$.

Thus, recalling~\eqref{5651}, we have
\begin{eqnarray*}
\int_{\R^n}\frac{w(y)}{|\zeta-y|^{n+2s}}\,dy  & = & 
-(-\Delta)^s w(\zeta)
\\
\\ & = &  -(-\Delta)^s u(\zeta',\zeta_n+\bar k) +(-\Delta)^s u_\star(\zeta)
\\
\\ & = &  W'(u(\zeta',\zeta_n+\bar k))
-W'(u_\star(\zeta))=0.\end{eqnarray*}
\vspace{1mm}

Since the integrand is nonnegative, this implies that~$w$ vanishes
identically, and so
$$
u(x',x_n+\bar k) =u_\star(x).
$$
Taking into account the above equality, \eqref{gsd78df1}
and the strict monotonicity of~$u$, it yields that
$$
u(x)<u_\star(x) \ \ \text{for any} \ x\in \R^n.
$$
This is in contradiction with the fact that~$u$
and~$u_\star$ coincide in~$\CC B_r$ and so Theorem~\ref{pro_min}
is proved.\qed
\end{proof}

\begin{remark}\label{rem_strict}
We note that hypothesis~\eqref{eq_strict} of strictly monotony in one direction in~Theorem~\ref{pro_min} may be relaxed as follows.
\begin{equation}\label{eq_relax}
\partial_{x_n} u(x) \geq 0, \ \ \text{for any} \ x \in \R^n.
\end{equation}
Precisely, if we assume that the function~$W$ belongs to~$C^2(\R)$, we can prove that the solutions of the Allen-Cahn equation~\eqref{eq56} that satisfy~\eqref{eq_relax} and~\eqref{lim56} are strictly increasing in one direction.

For this, we suppose by contradiction that $\partial_{x_n} u(\bar{x})=0$, for some~$\bar{x} \in \R^n$.
Then, by differentiating in~$x_n$ the equation in~\eqref{eq56},
we have that
\begin{eqnarray*}
-\int_{\R^n}\frac{
\partial_{x_n} u(y)
}{|\bar{x}-y|^{n+2s}}\,dy  & = &
\int_{\R^n}\frac{
\partial_{x_n}u(\bar x)-
\partial_{x_n}u(y)
}{|\bar{x}-y|^{n+2s}}\,dy
\\
\\ & = & 
(-\Delta)^s \partial_{x_n}u(\bar x)
\\
\\ & = & 
(-\Delta)^s \partial_{x_n}u(\bar x)+W''\big(u(\bar x)\big)
\,\partial_{x_n}u(\bar x)=0.\end{eqnarray*}
In view of~\eqref{eq_relax} the integrand is non-negative, and then we would obtain 
that~$\partial_{x_n}u$ vanishes identically. This would give
that~$u$ is constant along the $n$-direction, which is in contradiction
with~\eqref{lim56}. This proves that $u$ satisfies~\eqref{eq_strict}.
\end{remark}

\vspace{3mm}

\section{Minimizing the energy on intervals}\label{intermediate}
In this section, we deal with the problem of minimizing the energy $\FF$ on bounded intervals $I$ in $\R$.
\vspace{1mm}

The first result is in Lemma~\ref{lem_EN-1} below, in which we justify the existence of a minimizer $v_I$ and we provide a lower bound  for the corresponding energy $\FF(v_I, I)$ with respect to the length of $I$ (and depending on the fractional power $s$).

By Lemma~\ref{lem_EN-1}, together with some properties of the fractional Allen-Cahn equation~\eqref{equazione} proved in the Appendix, we will obtain an ulterior energy-estimate for the minimizers $v_I$ with $I=[0,R]$ and we will study
their asymptotic behavior as $R$ goes to $+\infty$ (see~Corollary~\ref{cor_alfa} and Corollary~\ref{cor_vmeno}).

\begin{lemma}\label{lem_EN-1}
Let~$I=[a,b]\subset\R$ be an interval with length~$|I|=b-a>4$.
Then, there exists
a measurable function~$v_I=v_{[a,b]}:\R\rightarrow[-1,+1]$
such that~$v_I(x)=-1$ if~$x\le a$, $v_I(x)=+1$
if~$x\ge b$ and
$$ \FF(v_I,\,I)\le\FF(v_I+\phi,\,I)$$
for any measurable function~$\phi$ supported in~$I$.

Moreover,
\begin{equation}\label{EN-I}
\FF(v_I,\,I)\le \Lambda(|I|):=
\begin{cases}
C_s(1+|I|^{1-2s}) & \text{if} \ s\in (0,\,1/2), \\
C_s(1+\log{|I|}) & \text{if} \ s = 1/2, \\
C_s & \text{if} \ s \in (1/2,\, 1),
\end{cases}
\end{equation}
for a suitable constant~$C_s>1$ depending only on $s$.
\end{lemma}

\begin{proof}
For any fixed interval $I=[a,b]\subset\R$, we will prove the existence of the function $v_I$, by means of Lemma~\ref{5656}. Therefore, 
it suffices to construct a suitable competitor in $[a,b]$. 
\vspace{1mm}

In view of translations invariance, we may suppose that $a<-2$ and $b>+2$. 
We consider the following function $h:\R\to[-1,1]$ defined by
$$
h(x):=
\begin{cases}
-1 &  \text{if} \ x\leq -1, \\
x & \text{if} \ x \in (-1, +1), \\
+1 & \text{if} \ x\geq +1,
\end{cases} 
$$

Let us compute each contribution in the energy $\FF(h, \, I)$. In the following, we will denote by $C_s$ suitable positive quantities, possibly different from line to line, and possibly depending on $s$.

\vspace{2mm}

First, since $h(x)\in \{-1,+1\}$ out of $(-1,1)$, we deduce
$$
\int_a^b W(h(x))\, dx \, = \, \int_{-1}^{+1} W(h(x)) \,dx \, \leq \, C_s,
$$
\vspace{1mm}

Second, let us compute the contributions in the kinetic term $\KK(h,\, I)$.
\vspace{1mm}

If $s\in (0,\,1/2)$,
$$
\int_a^{-1}\int_{+1}^b \frac{|h(x)-h(y)|^2}{|x-y|^{1+2s}}\,dx\, dy \, = \, 4 \int_a^{-1}\int_{+1}^b\frac{dx\,dy}{|x-y|^{1+2s}} \, \leq \, C_s|I|^{1-2s}.
$$

If $s=1/2$,
$$
\int_a^{-1}\int_{+1}^b \frac{|h(x)-h(y)|^2}{|x-y|^{1+2s}}\,dx\, dy \, = \, 4 \int_a^{-1}\int_{+1}^b\frac{dx\,dy}{|x-y|^{2}} \, \leq \, C_s\log{|I|}.
$$

If $s\in (1/2,\, 1)$,
\begin{eqnarray*}
\int_a^{-1}\int_{+1}^b \frac{|h(x)-h(y)|^2}{
|x-y|^{1+2s}}\,dx\, dy  =  4 \int_a^{-1}\int_{+1}^b\frac{dx\,dy}{
|x-y|^{1+2s}}\, \leq \, C_s.
\end{eqnarray*}

If $s\in(0,1)$,
\begin{eqnarray*}
\int_{-1}^{+1}\int_{-1}^{+1} \frac{|h(x)-h(y)|^2}{|x-y|^{1+2s}}\,dx\, dy  =  \int_{-1}^{+1}\int_{-1}^{+1}|x-y|^{1-2s}\, dx\, dy 
\, \leq \, C_s,
\end{eqnarray*}
\begin{eqnarray*}
\int_{-1}^{+1}\left[
\int_{+1}^{b} \frac{|h(x)-h(y)|^2}{|x-y|^{1+2s}}\,dy\right]\, 
dx 
= 
\int_{-1}^{+1}\left[
\int_{+1}^{b}\frac{|1-x|^2}{|x-y|^{1+2s}}\,dy\right]\, dx
\, \leq \, C_s
\end{eqnarray*}
and
\begin{eqnarray*}
\int_{-1}^{+1}\left[
\int_{a}^{-1} \frac{|h(x)-h(y)|^2}{|x-y|^{1+2s}}\,dy\right]\, 
dx
 = 
\int_{-1}^{+1}\left[
\int_{a}^{-1}\frac{|x+1|^2}{|x-y|^{1+2s}}\, dy\right]\, dx 
\, \leq \, C_s.
\end{eqnarray*}

Now,
we estimate the contribution coming from far in the energy:
\begin{eqnarray*}
\int\!\!\int_{[a,b]\times(\CC[a,b])} 
\frac{|h(x)-h(y)|^2}{|x-y|^{1+2s}}\,dx\, dy  & \leq & 
 \int\!\!\int_{[-1,1]\times(\CC[a,b])}
\frac{4}{|x-y|^{1+2s}}\,dx\, dy  
\\
&&+
\int\!\!\int_{[a,-1]\times [b,+\infty]}
\frac{4}{|x-y|^{1+2s}}\,dx\, dy  
\\ &&+
\int\!\!\int_{[1,b]\times(\CC(-\infty,a])}
\frac{4}{|x-y|^{1+2s}}\,dx\, dy  
\\
\\
&\le&
 \begin{cases}
C_s\big(1+|I|^{1-2s}\big) & \text{if} \ s\in (0,\,1/2), \\
C_s\big(1+\log{|I|}\big) & \text{if} \ s = 1/2, \\
C_s & \text{if} \ s \in (1/2,\, 1).
\end{cases} 
\end{eqnarray*}
All in all,
\begin{equation}\label{label}
\dys
\FF(h,\, I) \, \leq \,
\begin{cases}
C_s(1+|I|^{1-2s}) & \text{if} \ s\in (0,\,1/2), \\
C_s(1+\log{|I|}) & \text{if} \ s = 1/2, \\
C_s & \text{if} \ s \in (1/2,\, 1).
\end{cases} 
\end{equation}
Consequently,
we use Lemma~\ref{5656}
to obtain that there exists the desired function~$v_I$
such that~$v_I(x)=h(x)$ if~$x\in \CC[a,b]$ and~$\FF(v_I,\,I)\le
\FF(v_I+\varphi,\,I)$ for any measurable~$\varphi$ supported in $I$.
The estimate in~\eqref{EN-I} plainly follows from~\eqref{label}.\qed
\end{proof}
\vspace{1mm}

\begin{corollary}\label{cor_alfa}
Let the notation of Lemma~\ref{lem_EN-1} hold.
Then, fixed any~$\ell>0$, there exists a function~$\alpha_\ell
:(0,+\infty)\rightarrow(0,+\infty)$ such that
$$
\lim_{R\rightarrow+\infty} \alpha_\ell (R)=0
$$
and
\begin{equation}\label{eq_sul}
\FF(v_{[0,R]}, [-\ell,\ell]) \, \le \, \FF(v_{[0,R]}+\phi, [-\ell,\ell])
+\alpha_\ell(R).
\end{equation}
for any measurable function~$\phi$ supported in~$[-\ell,\ell]$. 
\end{corollary}

\begin{proof} The main idea is scaling the minimizer in~$[-\ell,R+\ell]$
in order to get a suitable competitor and then computing.
Let
$$
z_R(x)
:= v_{[-\ell,R+\ell]}\left( \frac{(R+2\ell)x}{R}-\ell\right).
$$
We observe that~$z_R(x)=-1$ if~$x\le0$ and~$z_R(x)=+1$
if~$x\ge R$. Therefore, by the minimality of~$v_{[0,R]}$, 
we get
\begin{equation}\label{eq_02e}
\FF(v_{[0,R]}, [0,R]) \, \leq \, \FF(z_R, [0,R]).
\end{equation}

Now, by scaling the variable of integration, we obtain
\begin{eqnarray}\label{eq_2l}
&&\FF(z_R, [0,R]) \nonumber \\
&& \  = \left(\frac{R}{R+2\ell}\right)^{\!1-2s}
\left[ \frac{1}{2}\int\!\!\int_{[-\ell, R+\ell]\times[-\ell, R+\ell]}\!\frac{|v_{[-\ell, R+\ell]}(x)-v_{[-\ell, R+\ell]}(y)|^2}{|x-y|^{1+2s}}\,dx\,dy
\right. \nonumber \\
&& \  \quad \left.  \! + \int\! \!\int_{[-\ell, 
R+\ell]\times(\CC[-\ell, R+\ell])} \frac{|v_{[-\ell, R+\ell]}(x)-v_{[-\ell, R+\ell]}(y)|^2}{|x-y|^{1+2s}}\,dx\,dy
\right]. \nonumber \\
&& \  \quad + \left(\frac{R}{R+2\ell}\right)\int_{-\ell}^{R+\ell}W\big(v_{[-\ell, R+\ell]}(x)\big)\, dx \nonumber \\
\nonumber \\
& & \ \leq \left(\frac{R}{R+2\ell}\right)^{\!1-2s} \FF(v_{[-\ell, R+\ell]}, {[-\ell, R+\ell]}).
\end{eqnarray}
\vspace{1mm}

Now, we set
$$
\beta_\ell(R):=
\begin{cases}
0 & \text{if} \ s\in (0,\,1/2], \\
\\
 \dys C_s\left[
\left(\frac{R}{R+2\ell}\right)^{\!1-2s}-1
\right] & \text{if} \ s\in (1/2\,, 1),
\end{cases}
$$
where $C_s$ is the constant introduced in~\eqref{EN-I}.
Notice that
$\dys \lim_{R\rightarrow+\infty} \beta_\ell(R)=0.$

In view of~Lemma~\ref{lem_EN-1}, we see that~\eqref{eq_2l} becomes 
\begin{eqnarray}\label{eq_35bis}
\FF(z_R, [0,R])  & \leq  &  \FF(v_{[-\ell, R+\ell]}, [-\ell, R+\ell]) \nonumber \\
&&  + \left[\left(\frac{R}{R+2\ell}\right)^{\!1-2s}-1
\right]\FF(v_{[-\ell, R+\ell]}, [-\ell, R+\ell])
\nonumber \\
\nonumber \\
& \leq &  \FF(v_{[-\ell, R+\ell]}, [-\ell, R+\ell])+\beta_\ell(R).
\end{eqnarray}
Hence, combining~\eqref{eq_35bis} with~\eqref{eq_02e}, we obtain, for any fixed $\ell>0$,
\begin{eqnarray}\label{eq_4lbis}
\FF(v_{[0,R]}, [0,R]) \, \leq \, \FF(v_{[-\ell, R+\ell]}, [-\ell, R+\ell])+\beta_\ell(R).
\end{eqnarray}

Now, we claim that, for any fixed $\ell>0$ there exists a function $\gamma_\ell:(0,+\infty)\to(0,+\infty)$ such that $\gamma_\ell(R)\to0$ as $R\to +\infty$ and
\begin{equation}\label{eq_5l}
\FF(v_{[0,R]},[0,R]) \, = \, \FF(v_{[0,R]}, [-\ell,R+\ell]) - \gamma_\ell(R).
\end{equation}
Indeed, since $v_{[0,R]}(x) = -1$ for any $x\leq 0$ and $v_{[0,R]}(x) = +1$ for any $x\geq R$, we have
\begin{eqnarray*}
&&\!\!\! \FF(v_{[0,R]},[-\ell,R+\ell])-\FF(v_{[0,R]},[0,R])
\\
&& \ \ \le
\, \int\!\!\int_{[-\ell,0]\times[R,+\infty)}
\frac{4}{|x-y|^{1+2s}}\,dx\,dy+\int\!\!\int_{[R,R+\ell]\times(-\infty,-\ell]}                     
\frac{4}{|x-y|^{1+2s}}\,dx\,dy
\\
&& \ \ \le \, \gamma_\ell(R),
\end{eqnarray*}
with
$$
\dys \gamma_\ell(R) \,:=\,\begin{cases}
\dys 4\log{\left(1+\frac{2\ell}{R}\right)} & \text{if} \ s=1/2,\\
\\
\dys \frac{2}{s(1-2s)}\Big((R+2\ell)^{1-2s}-R^{1-2s}\Big) & \text{if} \ s 
\neq 1/2.
\end{cases}$$
This gives~\eqref{eq_5l}.
\vspace{1mm}

{F}rom~\eqref{eq_4lbis} and~\eqref{eq_5l}, 
together with the minimality of the function $v_{[-\ell, R+\ell]}$, it follows
that \begin{equation}\label{eq_8l}
\FF(v_{[0,R]},[-\ell,R+\ell]) - \beta_{\ell}(R)-\gamma_\ell(R) \, \leq \, \FF(v_{[0,R]}+\phi, [-\ell, R+\ell]).
\end{equation}
\vspace{1mm}

Now, we use the fact that $\phi$ is supported in $[-\ell,\ell]$
to obtain
\begin{equation*}
\begin{split}
& \FF(v_{[0,R]}+\phi,[-\ell, R+\ell]) 
- \FF(v_{[0,R]} +\phi, 
[-\ell, \ell])
\\ &\ \, = \frac{1}{2}\int\!\!\int_{[\ell,R+\ell]\times[\ell,R+\ell]} 
\frac{|v_{[0,R]}(x)-v_{[0, R]}(y)|^2}{|x-y|^{1+2s}}\,dx\,dy
\\ & \quad \ \, + \int\!\!\int_{[\ell, R+\ell]\times(\CC[-\ell, 
R+\ell])}\frac{|v_{[0,R]}(x)-v_{[0, R]}(y)|^2}{|x-y|^{1+2s}}\,dx\,dy
+\int_{[\ell,R]} W(v_{[0,R]}(x))\,dx\\
&\ \, = \FF(v_{[0,R]},[-\ell, R+\ell]) 
- \FF(v_{[0,R]},[-\ell, \ell]).
\end{split}\end{equation*}
By plugging this identity into~\eqref{eq_8l}, we obtain
the desired 
result, with~$\alpha_\ell(R):=\beta_\ell(R)+\gamma_\ell(R).$\qed
\end{proof}

\vspace{1mm}

\begin{corollary}\label{cor_vmeno}
Let the notation of~Lemma~\ref{lem_EN-1} hold. Then the function 
$v_{[0,R]}$ converges to~$-1$
locally uniformly as $R$ goes to $+\infty$.
\end{corollary}

\begin{proof}
By minimality, the function $v_{[0,R]}$ is a solution of
$$
-(-\Delta)^s v_{[0,R]}(x)=W'(v_{[0,R]}(x)) \ \ \text{for any} \ x\in [0,R].
$$
Then, in view of Lemma~{\ref{lem_co}} and~\cite[Theorem 3.3]{CS10}, 
$v_{[0,R]}$ is uniformly continuous on the whole of $\R$ with modulus of 
continuity bounded independently of $R$ (recall the
footnote on 
page~\pageref{nota}). Hence, there exists a function $v$ such that, up to subsequences, $v_{[0,R]}\to v$ as $R\to+\infty$ locally uniformly in~$\R$. 
Moreover, by taking into account Lemma~\ref{lem_limiti}, the limit function $v$ satisfies
\begin{equation}\label{00}
-(-\Delta)^sv(x)= W'(v(x)) \ \ \ \text{for any} \ x \in [0,\infty).
\end{equation}
Also,
\begin{equation}\label{000}
\dys v(x) = -1 \ \ \text{for any} \ x \in (-\infty, 0].
\end{equation}
\vspace{2mm}

We claim that actually
\begin{equation}\label{MIN}
{\mbox{$v$ is a minimizer in the whole of $\R$,}}\end{equation}
i.e., for any $\ell>0$, 
$$
\FF(v, [-\ell, \ell]) \, \leq \, \FF(v+\phi, [-\ell, \ell])
$$
for any perturbation~$\phi$ supported in $[-\ell,\ell]$.
\vspace{1mm}

This fact follows by~Corollary~\ref{cor_alfa}. Indeed, we notice that
$$
\frac{|v_{[0,R]}(x)-v_{[0,R]}(y)|^2}{|x-y|^{1+2s}}\,
dx \, \leq \, 
\frac{4}{|x-y|^{1+2s}}
$$
and we have
\begin{eqnarray*}
&& \int\!\!\int_{[-\ell,\ell]\times(\CC[-2\ell,2\ell])}
\frac{4}{|x-y|^{1+2s}}\,dx\,dy \\
&& \qquad\qquad \qquad \qquad \qquad = \begin{cases}
\dys \frac{4}{s(1-2s)}\left[(3\ell)^{1-2s}-\ell^{1-2s}\right] & s\neq 1/2, \\
8\ln{3} & s=1/2
\end{cases}
 \, <  +\infty.
\end{eqnarray*}
Then the function $\dys \frac{4}{|x-y|^{1+2s}}$ belongs to $L^1\big( [-\ell,\ell]\times(\CC[-2\ell,2\ell])\big)$
and so, by the Dominated Convergence Theorem
\begin{eqnarray*}
&&\lim_{R\rightarrow+\infty}
\int\!\!\int_{[-\ell,\ell]\times(\CC[-2\ell,2\ell])}\!
\frac{|v_{[0,R]}(x)-v_{[0,R]}(y)|^2}{|x-y|^{1+2s}}\, 
dx\,dy\\
&&\qquad \qquad \qquad\qquad\qquad \qquad\qquad \ =
\int\!\!\int_{[-\ell,\ell]\times(\CC[-2\ell,2\ell])}\!
\frac{|v(x)-v(y)|^2}{|x-y|^{1+2s}}\,.\end{eqnarray*}
This and the uniform convergence of $v_{[0,R]}$ to $v$ imply
that
\begin{eqnarray*} && 
\lim_{R\to+\infty}\frac{1}{2}\!\int\!\!\int_{[-\ell,\ell]\times[-\ell,\ell]}\!\frac{|v_{[0,R]}(x)-v_{[0,R]}(y)|^2}{|x-y|^{1+2s}}\, 
dx\, dy \\
&& \qquad \qquad \qquad \qquad   \qquad \qquad \qquad
= \frac{1}{2}\!\int\!\!\int_{[-\ell,\ell]\times[-\ell,\ell]}\!\frac{|v(x)-v(y)|^2}{|x-y|^{1+2s}}\, 
dx\, dy\,,\\
\\
&&
\lim_{R\to+\infty}\int\!\!\int_{[-\ell,\ell]\times(\CC[-\ell,\ell])}\!\frac{|v_{[0,R]}(x)-v_{[0,R]}(y)|^2}{|x-y|^{1+2s}}\, 
dx\, dy \\
&& \qquad \qquad \qquad \qquad \qquad \qquad \qquad
= \int\!\!\int_{[-\ell,\ell]\times(\CC[-\ell,\ell])}\!\frac{|v(x)-v(y)|^2}{|x-y|^{1+2s}}\, 
dx\, dy \\ && \\ {\mbox{and }}&& 
\lim_{R\to+\infty}\int_{-\ell}^{\ell}W(v_{[0,R]}(x))\, dx \, =\, 
\int_{-\ell}^{\ell}W(v(x))\, dx.\end{eqnarray*} This implies that $$ 
\lim_{R\to+\infty} \FF(v_{[0,R]},[-\ell,\ell]) \, =\, \FF(v,[-\ell,\ell]) 
$$ and the same holds for the function~$v_{[0,R]}+\phi$,
i.e.,
$$
\lim_{R\to+\infty} \FF(v_{[0,R]}+\phi,[-\ell,\ell]) \, =\, \FF(v+\phi,[-\ell,\ell]).
$$
Then, by taking the limit as $R\to+\infty$ in~\eqref{eq_sul}, the claim 
in~\eqref{MIN} plainly follows.
\vspace{1mm}

Finally, by~\eqref{MIN} we see that~\eqref{00} holds for any $x\in\R$. This and~\eqref{000} yield that the function~$v$ is identically $-1$.\qed
\end{proof}
\vspace{1mm}

Next results may be seen as energy decreasing rearrangements with more elementary techniques with
respect to the ones in~\cite{garsia} (see also~\cite{Ba94} for more general integral inequalities).

\begin{lemma}\label{REA}
Given a measurable set~$\Omega$ and two measurable functions~$u$, 
$v:\Omega\rightarrow\R$. Then
\begin{equation}\label{ovi}
\FF(\min\{u,v\},\Omega) + \FF(\max\{u,v\},\Omega) \, \leq \, \FF(u,\Omega)+\FF(v,\Omega)
\end{equation}
and the equality holds if and only if
\begin{equation}\label{ovi2}
u(x) \leq v(x) \ \ \text{or} \ \ v(x) \leq u(x) \ \ \text{for any} \ x \in \Omega.
\end{equation}
\end{lemma}
\begin{proof}
Denote by
\begin{equation}\label{M m}
m(x):=\min\{u(x),v(x)\}\; {\mbox{
and }}\;
M(x):=\max\{u(x),v(x)\}.\end{equation}
Then, we may deduce the claim in~\eqref{ovi}-\eqref{ovi2} by the following fact. For any~$x$, $y\in\Omega$,
\begin{equation*}
|m(x)-m(y)|^2+|M(x)-M(y)|^2 \, \le \, |u(x)-u(y)|^2+|v(x)-v(y)|^2
\end{equation*}
and if equality holds then
\begin{equation}\label{S002a}
\big(u(x)-v(x)\big)\big(u(y)-v(y)\big)\ge0.
\end{equation}
This is straightforward to check and we leave the details to the reader.\qed
\end{proof}

\vspace{1mm}

\begin{corollary}\label{lem_momo}
Let the notation of~Lemma~\ref{lem_EN-1} hold. Then, $v_I$
is non-decreasing.
\end{corollary}

\begin{proof} First, we remark that for any function $w$ and $z$ such that $w=z$ outside $\Omega'\subseteq\Omega$, we have
\begin{equation}\label{outside}
\FF(w,\Omega)-\FF(z,\Omega) \, = \, \FF(w,\Omega')-\FF(z,\Omega').
\end{equation}

Now, for any~$\tau>0$, we let~
$$
u(x):=v_I(x) \ \ \text{and} \ \   v(x):=v_I(x+\tau)
$$
and we recall the setting in~\eqref{M m}.
We have that
\begin{equation}\label{CO1}{\mbox{
$M(x)=-1$ if~$x\le a-\tau$
and~$M(x)=+1$ if~$x\ge b-\tau$.}}
\end{equation}
Therefore, we can apply~\eqref{outside}, with $w:=M$ and $z:=v$ in $\Omega':=[a-\tau, b-\tau]\subseteq\Omega:=[a-\tau,b]$, and we get
\begin{eqnarray}\label{mono1}
\dys
&& \FF(M,[a-\tau,b])-\FF(v,[a-\tau,b]) \nonumber \\
&& \qquad \qquad \qquad \ =  \FF(M,[a-\tau,b-\tau]) -\FF(v,[a-\tau, b-\tau]) \geq 0,
\end{eqnarray}
where we also used the minimality of $v_I$.
\vspace{1mm}

Analogously, since
\begin{equation*}
{\mbox{
$m(x)=-1$ if~$x\le a$
and~$m(x)=+1$ if~$x\ge b$,}}
\end{equation*}
we get
\begin{eqnarray}\label{mono2}
\dys
\FF(m,[a-\tau,b])-\FF(u,[a-\tau,b]) \, \geq \, 0.
\end{eqnarray}
Consequently, by~\eqref{mono1} and~\eqref{mono2}, we conclude that
\begin{equation*}
\dys
\FF(m,[a-\tau,b])+\FF(M,[a-\tau,b]) \geq \FF(u,[a-\tau,b]) -\FF(v,[a-\tau, b).
\end{equation*}
Since we know that the reverse inequality holds true as well,
due to Lemma~\ref{REA}, we obtain that
$$ \FF(m,[a-\tau,b])+\FF(M,[a-\tau,b])
= \FF(u,[a-\tau,b])+\FF(v,[a-\tau,b]).$$
Therefore, by~\eqref{ovi2}, we have that~$u-v$ does not
change sign, hence~$v_I$ is monotone.\qed
\end{proof}

\vspace{3mm}
\section{The $1$-D minimizer - Proof of
Theorem~\ref{theorem_unod}}\label{sec_minimi}

We are ready to deal with the 1-D minimizers (for related observations 
when $s\in (1/2,\,1)$ see~\cite{garroni} and~\cite{gonzalez}).
\vspace{1mm}

\begin{proof}[Proof of Theorem~\ref{theorem_unod}]
For the sake of simplicity, we define the auxiliary set of functions $\mathcal{M}$ in $\mathcal{X}$ as follows
\begin{eqnarray}\label{def_m}
& {\mathcal{M}}  = \Big\{
u\in {\mathcal{X}} \st
{\mathcal{G}}(u)<+\infty \aand \FF (u,[-a,a])\le \FF(u+\phi,[-a,a]) \nonumber \\
&\quad \quad \ {\mbox{for any $a>0$ and
any $\phi$
measurable and supported in $[-a,a]$
}}\Big\}
\end{eqnarray}
and we divide the proof in few steps.

\noindent
\\ {\it Step 1.} Claim: the set~${\mathcal{M}}$ is non-empty.
\vspace{1mm}

We will prove this claim by taking the limit of a suitable sequence of functions in $\XX$.

\vspace{1mm}

For any $K>2$, we may use Lemma~\ref{lem_EN-1} with $a=-K$ and $b=K$ and we obtain a minimizer $v_{[-K,K]}:\R\to [-1,1]$ such that $v_{[-K,K]}(x)=-1$ if $x\leq 0$ and $v_{[-K,K]}(x)=+1$ if $x\geq K$. Also,
\begin{equation}\label{eq_TA}
\FF(v_{[-K,K]},[-K,K]) \leq 
\begin{cases}
C_s\big(1+(2K)^{1-2s}\big) & \text{if} \ s \in (0,\,1/2), \\
C_s\big(1+\log{(2K)}\big) & \text{if} \ s=1/2, \\
C_s & \text{if} \ s\in (1/2,\,1),
\end{cases}
\end{equation}
for a suitable constant $C_s>0$.

Also, we recall that, in view of~Corollary~\ref{lem_momo}, the function $v_{[-K,K]}$ is monotone non-decreasing.

\vspace{1mm}

The minimization property of~$v_{[-K,K]}$ yields that
\begin{eqnarray}\label{89865fffqyyqyq}
\int_{\R}\frac{v_{[-K,K]}(y)-v_{[-K,K]}(x)}{|x-y|^{1+2s}}\,dy & = &
-(-\Delta)^s 
v_{[-K,K]}(x) \nonumber \\
& = &  W'(v_{[-K,K]}(x)), 
\ \ \forall x\in [-K,K],
\end{eqnarray}
and so, by Lemma \ref{lem_co} and~\cite[Theorem 3.3]{CS10}, we have that~$v_{[-K,K]}$ is continuous, 
with modulus of continuity
bounded independently of~$K$.
\vspace{1mm}

Now, we fix a point~$c_o\in (-1,1)$
such that
\begin{equation}\label{co8d}
W'(c_o)\neq0.
\end{equation}
By continuity, there 
must be a point~$p_K\in[-K,+K]$
such that~$v_{[-K,K]}(p_K)=c_o$.

We claim that
\begin{equation}\label{bar}
\lim_{K\rightarrow+\infty} K-|p_K|=+\infty.
\end{equation}
To check this, we suppose by contradiction that there exists a constant $C>0$ such that $|p_K+K|\leq C$ for infinitely many $K$'s. We denote by $\dys p=\lim_{K\to +\infty} \left(p_K+K\right)$ and we consider the function $v_{[0,2K]}(x)=v_{[-K,K]}(x-K)$.
\vspace{1mm}

Notice that, according to Corollary~\ref{cor_vmeno}, $v_{[0,2K]}$ converges locally uniformly to $-1$ as $K\to+\infty$. 
Besides, for any  $x>p$, we have
$$
v_{[0,2K]}(x) \,=\, v_{[-K,K]}(x-K) \, \geq \, v_{[-K,K]}(p_K) \, =\,  c_o, \ \ \ (\text{for large} \ K)
$$
that implies
$$
\dys \lim_{K\to+\infty}v_{[0,2K]}(x) \, \geq\, c_o\, >\, -1
$$
and thus we get a contradiction. 
This proves~\eqref{bar}.

\vspace{2mm}

Now, we set
$$
u_K(x):=v_{[-K,K]}(x+p_K),
$$ 
so~$u_K(0)=c_o$. As a consequence,
we may suppose that~$u_K$ converges locally uniformly to 
some~$u_*
\in C(\R;[-1,+1])$, with
\begin{equation}\label{co9d}
u_*(0)=c_o
\end{equation}
and
\begin{equation}\label{19011}
{\mbox{
$u_*$
is non-decreasing.}}\end{equation}
By~\eqref{89865fffqyyqyq} and Lemma~\ref{lem_limiti},
\begin{equation}\label{os11}
{\mbox{$(-\Delta)^s u_*(x)+W'(u_*(x))=0$ \ for any $x\in\R$.}}
\end{equation}
This and Lemma \ref{lem_dif} imply that~$u_*\in C^2(\R)$. 
{F}rom~\eqref{19011}, we already know that~$u_*'\ge 0$, and then, by arguing as in~Remark~\ref{rem_strict}, one can prove that
\begin{equation}\label{str}
{\mbox{$u_*'(x)>0$
for any $x\in\R.$}}
\end{equation}

\vspace{1mm}

Now, we prove that
\begin{equation}\label{os23} {\mathcal{G}}(u_*)<+\infty.
\end{equation}
Indeed, by~\eqref{eq_TA}, we get
\begin{eqnarray*}
\FF(u_K, [p_K-K,p_K+K])  & = &
\FF(v_{[-K,K]},[-K,+K])
\\ & \le & 
\begin{cases}
C_s\big(1+
(2K)^{1-2s}\big) &\text{if} \ s\in(0,\,1/2), \\
C_s\big(1+
\log{(2K)}\big) &\text{if} \ s=1/2, \\
C_s &\text{if} \ s\in(1/2,\, 1), \\
\end{cases}
\end{eqnarray*}
This, \eqref{bar} and Fatou Lemma imply~\eqref{os23}.
\vspace{2mm}

Moreover, $u_{*}$ is such that
\begin{equation}\label{os12}
\lim_{x\rightarrow\pm \infty} u_*(x)=\pm1.
\end{equation}
We can prove~\eqref{os12} arguing by contradiction. By~\eqref{str},
we know that there exists~$a_-$, $a_+$ such that
$$ -1\le a_-<a_+\le+1$$
and
$$ \lim_{x\rightarrow \infty} u_*(x)=a_\pm.$$
Let us show that~$a_-=-1$. Suppose, by contradiction, that
\begin{equation}\label{55677}
a_->-1.\end{equation}
Then, we set~$a_*:=(a_-+a_+)/2\in(-1,a_+)$ and we infer from~\eqref{55677}
that 
$$ i:=\inf_{[a_*,a_+]} W>0.$$
Recalling~\eqref{19011}, 
we have that there exists~$\kappa\in\R$ such that, if~$x\ge \kappa$,
then~$u_*(x)\in [a_*,a_+]$.
So,
from~\eqref{os23},
\begin{eqnarray*}
+\infty  & > &  {\mathcal{G}}(u_*) \\
& \ge   &  \lim_{R\rightarrow+\infty}\begin{cases}\dys  R^{2s-1}\int_{\kappa}^{R} W(u_*)\,dx & \text{if} \ s\in (0,\,1/2)\\
\dys (\log R)^{-1}\int_{\kappa}^{R} W(u_*)\,dx & \text{if} \ s =1/2\\
\dys \int_{\kappa}^{R} W(u_*)\,dx & \text{if} \ s\in (1/2,\,1)
\end{cases}\\
&  \ge &    \lim_{R\rightarrow+\infty}i\,(R-\kappa)\begin{cases}\dys R^{2s-1}& \text{if} \ s\in (0,\,1/2)\\
\dys (\log R)^{-1} & \text{if} \ s =1/2\\
1 & \text{if} \ s\in (1/2,\,1)
\end{cases} \ = \  +\infty,
\end{eqnarray*}
and this contradiction proves
that~$a_-=-1$. Analogously, one proves that~$a_+=+1$.
This finishes the proof of~\eqref{os12}.
\vspace{1mm}
\vspace{1mm}

By~\eqref{os12} and Theorem~\ref{pro_min},
we obtain that
\begin{equation}\label{os33}
\begin{split}&
\FF (u_*,[-a,a])\le \FF(u_*+\phi,[-a,a]) \\
&\qquad \qquad {\mbox{ for any $a>0$ and
any $\phi$
measurable and supported in $[-a,a]$
.}}\end{split}\end{equation}

By collecting the results in~\eqref{os23},
\eqref{os12} and~\eqref{os33}, we obtain that the set $\MM$ is not empty.

\vspace{1mm}

Now, for any~$x_o\in\R$, define the set
\begin{equation}\label{def_mx}
{\mathcal{M}}^{(x_o)}:=\big\{
u\in {\mathcal{M}} \st x_o=\sup\{t\in\R\st u(t)<0\}\,\big\}.
\end{equation}

\vspace{2mm}

\noindent
{\it Step 2.} Claim: the set $\MM^{(x_o)}$ consists of only one element, which will be denoted by $u^{(x_o)}$, and $u^{(x_o)}(x)=u^{(0)}(x-x_o)$. 
\vspace{1mm}

Now, we prove that there exists~$x_*\in\R$ such that
\begin{equation}\label{albo}
{\mbox{${\mathcal{M}}^{(x_*)}$ has only one 
element.}} \end{equation}
For this, we consider the previously constructed minimizer~$u_*$
and we take~$x_*\in\R$ such that ~$u_*\in {\mathcal{M}}^{(x_*)}$.
Let us take $u\in {\mathcal{M}}^{(x_*)}$. By cutting at the levels~$\pm 
1$,
we see that~$|u|\le 1$. Thus, for any fixed~$\eps>0$, there 
exists~$k(\eps)\in\R$ such that, for~$k\in(-\infty,\, k(\eps)]$, 
we have 
$$
u(x-k)+\eps>u_*(x) \ \ \text{for any} \ x\in\R.
$$
Now we take~$k$ as large as possible with the above property; that is,
we take~$k_\eps$ such that
\begin{equation}\label{KKS}
u(x-k_\eps)+\eps\ge u_*(x)\end{equation}
for any~$x\in\R^n$ and, for any~$j\ge 1$ there exist
a sequence~$\eta_{j,\eps}\ge0$ and points~$x_{j,\eps}\in\R$
such that
$$\lim_{j\rightarrow+\infty}\eta_{j,\eps}\,=\,0$$
and~$u(x_{j,\eps}-(k_\eps+\eta_{j,\eps}) )+\eps\le u_*(x_{j,\eps})$.

We observe that~$x_{j,\eps}$ must be a bounded sequence in~$j$. Otherwise,
if
$$ \lim_{j\rightarrow+\infty} x_{j,\eps}=\pm\infty,$$
then
$$ \pm1+\eps\,  =\,  \lim_{j\rightarrow+\infty}
u(x_{j,\eps}-(k_\eps+\eta_{j,\eps}))+\eps\,  \le \,  
\lim_{j\rightarrow+\infty} u_*(x_{j,\eps})\,  =\,  \pm1,$$
which is a contradiction.

Therefore, we may suppose that
$$ \lim_{j\rightarrow+\infty} x_{j,\eps}=x_\eps,$$
for some~$x_\eps\in \R$.
By~\eqref{os11} and by Lemma \ref{lem_co}, we know that~$u$ and~$u_*$ are continuous (recall \ref{5651}),
therefore
\begin{equation}\label{STO}
u(x_{\eps}-k_\eps)+\eps = u_*(x_{\eps}).\end{equation}
Thus, if we set
$$
u_\eps(x):=u(x-k_\eps)+\eps,
$$
we have that~$u_\eps\ge 
u_*$, $u_\eps(x_\eps)=u_*(x_\eps)$ and, by~\eqref{os11},
$$ 
-(-\Delta)^s
u_\eps (x) \,  = \,  -(-\Delta)^s u(x-k_\eps) \, =  \, 
W'(u(x-k_\eps)) \, = \, W'(u_\eps(x)-\eps).$$
Consequently,
\begin{eqnarray}\label{STI}
 0  & \le & 
\int_\R
\frac{
(u_\eps-u_*) (y)
}{|x_\eps-y|^{1+2s}}\,dy \ = \ -(-\Delta)^s (
u_\eps-u_*)(x_\eps) \nonumber 
\\ \nonumber
\\ &= &  W'(u_*(x_\eps)-\eps)-W'(u_*(x_\eps)).
\end{eqnarray}
\vspace{1mm}

Now, we claim that
\begin{equation}\label{BB1}{\mbox{
$|x_\eps|$ is bounded.}}\end{equation}
Indeed, suppose that, for some subsequence,
$$
\lim_{\eps\rightarrow0^+}|x_\eps|=+\infty.
$$
Then,
\begin{equation}\label{eq_limu}
\lim_{\eps\rightarrow0^+} u_*(x_\eps)=\pm 1.
\end{equation}
By taking into account hypothesis \eqref{CON6} on the potential $W$, we have that
\begin{equation}\label{grow33}
{\mbox{$W'(t)\ge W'(r)+c(t-r)$ when
$r\le t$, $r,t\in [-1,\,-1+c]\cup[+1-c,\,+1]$,
}}
\end{equation}for some~$c>0$.
\vspace{1mm}
Then, by \eqref{eq_limu} there exists~$\eps_o>0$ such that
both~$u_*(x_\eps)$ 
and~$u_*(x_\eps)-\eps$
belong, for~$\eps\in (0,\eps_o)$, to~$[-1,\,-1+c]\cup[+1-c,\,+1]$,
where~$c>0$ is the one given by~\eqref{grow33}.
It follows
$$
W'(u_*(x_\eps))\, \ge\,  W'(u_*(x_\eps)-\eps)+c\eps\, >\, 
W'(u_*(x_\eps)-\eps),
$$
and this is in contradiction with~\eqref{STI}. Thus~\eqref{BB1} is proved.

As a consequence, we may suppose, up to subsequences, that
\begin{equation}\label{STU}
\lim_{\eps\rightarrow0^+} x_\eps=x_o,
\end{equation}
for some~$x_o\in\R$.
\vspace{1mm}

We also have that
\begin{equation}\label{BB2}{\mbox{
$|k_\eps|$ is bounded.}}\end{equation}
Indeed, if
$$ \lim_{\eps\rightarrow0^+}k_\eps=\pm \infty,$$
we would obtain from~\eqref{STO} and~\eqref{STU} that
$$ \mp 1\,=\,
\lim_{\eps\rightarrow0^+}u(x_\eps-k_\eps)+\eps\,=\,
 \lim_{\eps\rightarrow0^+}u_*(x_\eps)\,=\, u_*(x_o),$$
and so, from~\eqref{os11},
$$0\, =\, W'(u_*(x_o))\, =\, -(-\Delta)^s u_*(x_o)\, =\, 
\int_\R\frac{u(y)\pm1}{|x_o-y|^{1+2s}}\,dy.$$
Since the integrand is either non-negative or non-positive,
it follows that~$u_*$ is identically equal to~$\pm1$,
which is a contradiction. This proves~\eqref{BB2}.
\vspace{1mm}

Accordingly, we may suppose that
$$ \lim_{\eps\rightarrow0^+} k_\eps=k_o,$$
for some~$k_o\in\R$. Hence,
$$ \lim_{\eps\rightarrow0^+} (u_\eps-u_*)(y)=
\lim_{\eps\rightarrow0^+} u(y-k_\eps)+\eps-u_*(y)=
u(y-k_o)-u_*(y), \ \ \ \forall y \in \R,
$$
and so, passing to the limit in~\eqref{STI},
we conclude that
\begin{equation}\label{STK}
\int_\R
\frac{ 
u(y-k_o)-u_* (y)
}{|x_\eps-y|^{1+2s}}\,dy\,=\,0.
\end{equation}
On the other hand, by passing to the limit in~\eqref{KKS}, we see that~$
u(x-k_o)\ge u_*(x)$ for any~$x\in\R$, that is, the integrand 
in~\eqref{STK} is non-negative.
Consequently,
\begin{equation}\label{EEQ}
{\mbox{
$u_*(x)=u(x-k_o)$ for any~$x\in\R$.}}\end{equation}
\vspace{1mm}

We claim that
\begin{equation}\label{ko}
k_o=0.\end{equation}
To check this, we argue as follows.
Since~$u$ belongs to ${\mathcal{M}}^{(x_*)}$, we have that
$${\mbox{ if
$u(x)<0$ then $x\le x_*$,}}$$
and that
$$ {\mbox{there exists an infinitesimal
sequence~$\eps_j>0$ such that $u(x_*-\eps_j)<0.$}}$$
Hence, by~\eqref{EEQ},
\begin{equation}\label{87.1}
{\mbox{ if
$u_*(x)<0$ then $x\le x_*+k_o$}}\end{equation}
and
\begin{equation}\label{90.1} {\mbox{there exists an infinitesimal 
sequence~$\eps_j>0$ such that $u_*(x_*+k_o-\eps_j)<0.$}}
\end{equation}
On the other hand, since~$u_{*}\in{\mathcal{M}}^{(x_*)}$, we have that
\begin{equation}\label{90.2}{\mbox{ if
$u_*(x)<0$ then $x\le x_*$}}\end{equation}
and
\begin{equation}\label{87.2}
{\mbox{there exists an infinitesimal
sequence~$\delta_j>0$ such that $u_*(x_*-\delta_j)<0.$}}\end{equation}
By~\eqref{90.1} and~\eqref{90.2}, we have that~$x_*+k_o-\eps_j\le x_*$
and so, by passing to the limit,~$k_o\le 0$.
But, from~\eqref{87.1} and~\eqref{87.2}, we have that~$x_* -\delta_j\le 
x_*+k_o$,
that is, again by passing to the limit, 
$k_o \geq 0$.

The observations above prove~\eqref{ko}, that is $k_o=0$.
Then, from~\eqref{EEQ} and~\eqref{ko}, we have that~$u=u_*$,
and this proves~\eqref{albo}.
\vspace{2mm}

From \eqref{albo} we can easily deduce that the set $\mathcal{M}^{(x_o)}$ consists of only one element,  for any $x_o\in \R$.

Take any $u\in \mathcal{M}^{(x_o)}$ and set $\tilde{u}(x)=u(x+(x_{*}-x_o))$ for every $x\in \R$. Since such translate function $\tilde{u}$ belongs to $\mathcal{M}^{(x_*)}$, it follows that $\tilde{u}\equiv u_*$.
Accordingly, $u\in \mathcal{M}^{(x_o)}$ is such that $u(x)=u_*(x-(x_*-x_o))$; i.e., $\mathcal{M}^{(x_o)}$ consists of only one element. By the arbitrariness of $x_o\in \R$, the claim in Step~2 is proved.

\vspace{2mm}

\noindent
{\it Step 3.} Claim: $u^{(0)} \in C^2(\R)$ is such that $(u^{(0)})'(x)>0$ for any $x\in \R$ and $\MM^{(x_0)}\equiv\{u\in \MM \ \text{s.t.}\  u(x_o)=0\}$.
\vspace{1mm}

First, in view of~\eqref{os11} and the regularity assumptions on the function $W$, by Lemma~\ref{lem_dif} we can deduce that $u^{(0)}$ belongs to $C^{2}(\R)$.
Moreover, we know from the previous step that~${\mathcal{M}}^{(0)}$
only consists of one element and, in the proof of the claim in Step~1, we built one with positive
derivative (recall~\eqref{str}).
In particular such $u^{(0)}$ is continuous and strictly monotone increasing.
\vspace{2mm}

Finally, we observe that, in view of the previous steps, the minimum $u^{(0)}$ satisfies the hypothesis in~Proposition~\ref{pro_control} and then the estimates in~\eqref{sss} plainly follow.

\vspace{2mm}

The proof of Theorem~\ref{theorem_unod} is complete.\qed
\end{proof}

\vspace{1mm}

\begin{remark}{\it Existence of global minimizers in the case $s\in (1/2,\,1)$.}

We note that when $s\in (1/2,\,1)$ the functional $\GG$ coincides with $\FF$ on $\XX$. Hence, in view of Theorem~\ref{pro_min} and the fact that global minimizers of $\FF$ are solutions of the equation (\ref{equazione}), we can provide an alternative proof of the existence result in~Theorem~\ref{theorem_unod}, by  showing the existence of a monotone global minimizer which satisfies the limit condition~\eqref{lim56}. We will prove that the following infimum 
\begin{eqnarray}\label{profile}
\gamma_1 :=\inf\left\{\GG(v), \ v:\R\to\R \, \ \text{s.t.} \, \lim_{x\to\pm\infty}v(x)=\pm1 \right\}\,
\end{eqnarray}
is achieved by an non-decreasing function.

The key of the proof is given by the fact that the energy functional $\GG$ is decreasing with respect to monotone rearrangements. The proof is adapted from \cite[Theorem 2.4]{alberti}, in which the authors deals with a nonlocal functional deriving from Ising spin systems.
\vspace{1mm}

First, we recall that the energy $\GG$ is also decreasing under truncations by $-1$ and $+1$ and then it is not restrictive to minimize the problem (\ref{profile}) with the additional condition $|u|\leq 1$.
\vspace{1mm}

We denote by $X$ the class of all $v:\R\to[\-1,1]$ such that $\dys \lim_{x\to\pm\infty}v(x)=\pm1$; we denote by $X^{\star}$ the class of $v\in X$ such that $v$ is non-decreasing and $\dys v(0)=0$.
\vspace{1mm}

We claim that the infimum of $\GG$ on $X$ is equal to the infimum of $\GG$ on $X^{\star}$.
\vspace{1mm}

In fact, since $X^{\star}\subset X$ we have $\dys \inf_{v\in X^{\star}}\GG(v)\geq\inf_{v\in X}\GG(v)$, while the reverse inequality follows mainly by the fact that the singular perturbation term in the energy $\GG$ is decreasing under monotone rearrangements; see for instance~\cite[Theorem 9.2]{AL89} (and  also~\cite[Theorem~2.11]{alberti2000})
and \cite[Theorem I.1]{garsia} for monotonicity on the real line and on bounded intervals, respectively.\vspace{1mm}

Now, we are in position to show that the infimum of $\GG$ on $X^{\star}$ is achieved, by the direct method.
\vspace{1mm}

Take a minimizing sequence $(u_n)\subset X^{\star}$.
Since $u_n$ is non-decreasing and converging to $-1$ and $+1$ at $\pm\infty$, its distributional derivative $u'_n$ is a positive measure on $\R$ with $\|u'_n\|
=|Du_n(\R)|=2 
<+\infty, \ \forall n\in\N$.
Then there exist $u_\ast\in BV_{\text{loc}}(\R)$ and a subsequence $(u_{n_k})$ such that $u_{n_k}$ converges to $u_\ast$ almost everywhere as $k$ goes to $+\infty$ (see for instance \cite[Helly's First Theorem]{carothers}). By construction, $u$ is non-decreasing and satisfies 
$\dys u_\ast(x)=0$.
\vspace{1mm}

Let us show that $\dys \lim_{x\to\pm\infty}u_\ast(x)=\pm1$.

Since $u_\ast$ is non-decreasing in $[-1,1]$, there exist $a<0$ and $b>0$ such that
$$
\dys \lim_{x\to-\infty}u_\ast(x)=a \ \ \ \text{and} \ \ \ \lim_{x\to+\infty}u_\ast(x)=b.
$$
By contradiction, we assume that either $a\neq-1$ or $b\neq 1$. Then, since $W$ is continuous and strictly positive in $(-1,1)$, we obtain
$$
\dys \int_{\R}W(u_\ast)\,dx=+\infty.
$$
This is impossible, because, by Fatou's Lemma, we have
\begin{equation}\label{eq_fatou}
\dys \int_\R W(u)\,dx \, \leq\,  \liminf_{n\to+\infty}\int_{\R} W(u_n)\,dx \, \leq\,  \liminf_{n\to+\infty}\GG(u_{n})<+\infty.
\end{equation}
Hence, $u_\ast$ belongs to $X^{\star}$.

\vspace{1mm}

Finally, since $\GG$ is lower semicontinuous on sequences such that $u_n\to u_\ast$ pointwise,
the minimum problem $\gamma_1$ has a solution and this concludes the proof.
\vspace{2mm}

It is worth mentioning that an ulterior proof of the existence of minimizers for~\eqref{profile}
can be found in~\cite{garroni}, where it was studied the 1-D functional $\tilde{\FF}$ given by
$$
\dys \tilde{\FF}(u) = \int_\R\int_{\R}\frac{|u(x)-u(y)|^p}{|x-y|^p}dx\, dy+\int_{\R}W(u)\,dx, \ \ \ \ \ (p>2).
$$
Our case is analogous if we take $p=1+2s \in (2,3)$, since the exponent of the term ${|u(x)-u(y)|}$ does not play any special role in the proof (see~\cite[Proposition 3.3]{garroni}).
\end{remark}

\vspace{3mm}

\section{Extending the $1$-D minimizer to any
dimension - Proof of Theorem~\ref{theorem_0956}}\label{sd97qq1}

We start by proving the following lemma, which we will need 
in the proof of Theorem~\ref{theorem_0956}:

\begin{lemma}\label{lem_aux}
Let~$K\ge0$.
Let~$\alpha :\R\times [1,+\infty)\rightarrow [0,K]$
and~$\beta:\R\rightarrow [0,+\infty)$.
Suppose that~$\alpha(\cdot,R)$ is measurable for any fixed~$R\in
[1,+\infty)$ and that~$\beta$ is measurable.

Let also~$\lambda_R\in 
(0,+\infty)$ for any~$R\in [1,+\infty)$.

Assume that
\begin{eqnarray}
\label{EE0}
&& {\mbox{for any $\eta\in(0,1)$, }}\;
\lim_{R\rightarrow+\infty}\frac{\lambda_R}{\lambda_{\eta R}}=1,\\
\label{EE1}
&& {\mbox{for any $\eta\in(0,1)$, }}\;
\lim_{R\rightarrow+\infty} \sup_{|t|\le \eta R}\alpha(t,R)=0,\\
\label{EE2}
&& {\mbox{for any $R\ge K$, }}\;\lambda_R \int_{-R}^R \beta(t)\,dt
\le K,\\
{\mbox{and }}&& \label{EE3}
\liminf_{R\rightarrow+\infty}\lambda_R \int_{-R}^R \beta(t)\,dt=c,
\end{eqnarray}
for some~$c\in\R$. Then
$$ \liminf_{R\rightarrow+\infty}\lambda_R \int_{-R}^R
\alpha(t,R)\,\beta(t)\,dt=0.$$
\end{lemma}

\begin{proof} We fix~$\eta\in(0,1)$
and we use~\eqref{EE3} and~\eqref{EE0} to see 
that
\begin{eqnarray*}
c&=&
\liminf_{R\rightarrow+\infty}\lambda_R \int_{-R}^R \beta(t)\,dt
\\ &=&
\liminf_{R\rightarrow+\infty}\left(\lambda_R \int_{-\eta R}^{\eta R} 
\beta(t)\,dt
+ \lambda_R \int_{\{\eta R<|t|\le R\}}
\beta(t)\,dt\right)
\\ &\ge&
\liminf_{R\rightarrow+\infty}\lambda_R \int_{-\eta R}^{\eta R}
\beta(t)\,dt
+ \liminf_{R\rightarrow+\infty} \lambda_R \int_{\{\eta R<|t|\le R\}}
\beta(t)\,dt
\\ &=&
\liminf_{R\rightarrow+\infty}\lambda_{\eta R} \int_{-\eta R}^{\eta R}
\beta(t)\,dt
+ \liminf_{R\rightarrow+\infty} \lambda_R \int_{\{\eta R<|t|\le R\}}
\beta(t)\,dt\\ &=&
c+ \liminf_{R\rightarrow+\infty} \lambda_R \int_{\{\eta R<|t|\le R\}}
\beta(t)\,dt.
\end{eqnarray*}
As a consequence, by simplifying~$c$,
$$ 0\ge \liminf_{R\rightarrow+\infty} \lambda_R \int_{\{\eta R<|t|\le
R\}}\beta(t)\,dt.$$
So, since the integrand is non-negative,
\begin{equation}\label{FF}
\liminf_{R\rightarrow+\infty} \lambda_R \int_{\{\eta R<|t|\le 
R\}}\beta(t)\,dt=0.\end{equation}
Now, we use \eqref{EE2}, \eqref{FF} and~\eqref{EE1}
to conclude that
\begin{eqnarray*}
&& \liminf_{R\rightarrow+\infty}\lambda_R \int_{-R}^R \beta(t)
\alpha(t,R)\,dt
\\ &&\qquad\le
\liminf_{R\rightarrow+\infty} \left[\sup_{|\tau|\le\eta R}\alpha(\tau,R)
\;\lambda_R \int_{-\eta R}^{\eta R} \beta(t)\,dt
+K\lambda_R
\int_{\{\eta R<|t|\le R\}}\beta(t)\,dt\right]
\\ &&\qquad\le
\liminf_{R\rightarrow+\infty} \left[
K \sup_{|\tau|\le\eta R}\alpha(\tau,R)
+K \lambda_R \int_{\{\eta R<|t|\le R\}}\beta(t)\,dt
\right]\\ &&\qquad=
\lim_{R\rightarrow+\infty} K \sup_{|\tau|\le\eta R}\alpha(\tau,R)
+\liminf_{R\rightarrow+\infty}
K\lambda_R \int_{\{\eta R<|t|\le R\}}\beta(t)\,dt
\\ &&\qquad=0,
\end{eqnarray*}
which implies the desired result.\qed
\end{proof}

\vspace{1mm}

\begin{proof}[Proof of Theorem~\ref{theorem_0956}]
First, we recall that, by construction, the function $u^\ast$ defined in~\eqref{defust} 
coincides with the $1$-D minimizer $u^{(0)}$
 along the $n$-th coordinate $x_n$. Then, Theorem~\ref{theorem_unod} yields
\begin{equation}\label{eq_gp1}
\partial_{x_n} u^{\ast}(x) \, =\,  (u^{(0)})'(x_n) \, > \, 0 \ \ \ \forall x \in \R^n
\end{equation}
and 
\begin{equation}\label{eq_gp2}
\dys \lim_{x_n\to \pm\infty} u^{\ast}(x', x_n) =\lim_{x_n\to\pm \infty} u^{(0)}(x_n) \, = \, \pm 1 \ \ \ \forall x'\in \R^{n-1}.
\end{equation}
In view of~\eqref{eq_gp1} and~\eqref{eq_gp2}, it remains to show that $u^\ast$ satisfies $-(-\Delta)^s u^\ast(x) = W'(u^\ast(x))$, for any $x\in \R^n$, and \eqref{8swqkfhfee} will follow by Theorem~\ref{pro_min}. This is straightforward, since, by setting
\begin{equation}\label{def_newvar}
z':=(y'-x')/|y_n-x_n|  \ \ \ \text{and} \ \ \  z_n:=\varpi y_n
\end{equation}
the change of variable formula yields
\begin{eqnarray*}
-(-\Delta)^s u^*(x)
&=& \int_\R \left[ \int_{\R^{n-1}} \frac{
u^{(0)}(\varpi y_n)-u^{(0)}(\varpi x_n)
}{|x_n-y_n|^{n+2s}\,\left( 
1+\displaystyle\frac{|x'-y'|^2}{|x_n-y_n|^2}\right)^{(n+2s)/2} }
\,dy'
\right]dy_n
\\ \\ &=&
\varpi^{2s}
\int_\R\left[
\int_{\R^{n-1}} \frac{u^{(0)}(z_n)-u^{(0)}(\varpi x_n)}{|\varpi 
x_n-z_n|^{1+2s}\,
(1+|z'|^2)^{(n+2s)/2}}\, dz'
\right]dz_n
\\ \\ &=& \int_{\R} \frac{u^{(0)}(z_n)-u^{(0)}(\varpi x_n)}{|\varpi
x_n-z_n|^{1+2s}}\,dz_n \ = \ W'( u^{(0)}(\varpi x_n))
\\ \\ &=& W'(u^*(x)).
\end{eqnarray*}

\vspace{2mm}

Now, we will prove the claims in~(i), (ii) and~(iii). 

We need to carefully estimate the contribution on $B_R$ and on $\CC B_R$ of the $H^s_0$ norm of the function $u^\ast$.

Let $s\in (0,1)$, we observe that by the estimate in~\eqref{sss} it follows that there exists a constant $C_1>0$ such that
$$
\|(u^{(0)})'(x_n)\|_{L^\infty 
\big([x_n-(|x_n|/2),x_n+|x_n|/2]\big)}\le
C_1 |x_n|^{-(1+2s)} 
$$
for any $x_n$ large enough.

Accordingly, Lemma~\ref{lip_lem-0} (used here with~$\rho:=|x_n|/2$)  gives
\begin{equation}\label{09u1} \int_\R \frac{
|u^{(0)}(x_n)-u^{(0)}(y_n)|^2
}{|x_n-y_n|^{1+2s} }\,dy_n \le C_2 |x_n|^{-2s},
\end{equation}
for any~$x_n\in\R^n$ with~$|x_n|$ large enough, for a suitable constant~$C_2>0$.
\vspace{1mm}

{F}rom~\eqref{09u1},
we obtain that, for any~$x\in\R^n$ with~$|x_n|$ large enough,
\begin{eqnarray}\label{sd88dfsoo3ii31}
\int_{\R^n } \frac{|u^*(x)-u^*(y)|^2}{|x-y|^{n+2s} }\,dy  & \le & C_3\int_\R \frac{|u^{(0)}(\varpi x_n)-u^{(0)}(\varpi y_n)|^2
}{|x_n-y_n|^{1+2s} }\,dy_n \nonumber
\\ \nonumber 
\\ & \le &  C_4 |x_n|^{-2s},\end{eqnarray}
for suitable~$C_3$, $C_4>0$.
\vspace{1mm}

Also, if~$x\in\R^n$ with~$|x_n|\le R/2$,
we have that
\begin{equation}\label{sd88dfsoo3ii32}
\int_{\CC B_R}
\frac{
|u^*(x)-u^*(y)|^2}{|x-y|^{n+2s} }\,dy \le
\int_{\CC B_R}
\frac{
4
}{(|y|/2)^{n+2s} }\,dy\le C_5 R^{-2s}
\end{equation}
for a suitable~$C_5>0$.
\vspace{1mm}

Hence, for any~$R\ge4$, by~\eqref{sd88dfsoo3ii31}
and~\eqref{sd88dfsoo3ii32}, we get
\begin{eqnarray}\label{eq_gp3}
&& \int_{B_R}\int_{\CC B_R}
\frac{|u^*(x)-u^*(y)|^2}{|x-y|^{n+2s}}\,dx\,dy
\nonumber\\
\nonumber\\
&& \qquad\qquad \le
\int_{B_R\cap \{|x_n|
\le R/2\}
}\int_{\R^n}
\frac{|u^*(x)-u^*(y)|^2}{|x-y|^{n+2s}}\,dx\,dy
\nonumber\\
&&\quad \qquad \qquad+
\int_{B_R\cap \{|x_n|>R/2\} }\int_{\CC B_R}
\frac{|u^*(x)-u^*(y)|^2}{|x-y|^{n+2s}}\,dx\,dy
\nonumber\\
\nonumber\\
&& \qquad \qquad \le C_5\int_{B_R\cap \{|x_n|
\le R/2\}
} R^{-2s}\,dx
+C_4
\int_{B_R\cap \{|x_n|>R/2\} }
|x_n|^{-2s}\,dx
\nonumber\\
\nonumber\\
&&  \qquad \qquad \le
C_6\,
R^{n-2s}
,\end{eqnarray}
for a suitable~$C_6>0$.
\vspace{1mm}

Note that by~\eqref{eq_gp3} it follows
\begin{eqnarray*}
&&\dys \text{if} \ s =1/2, \\
&&\qquad \qquad \frac{1}{R^{n-1}\log R}\int_{B_R}\int_{\CC B_R}\!\!\frac{|u^*(x)-u^*(y)|^2}{|x-y|^{n+2s}}dx\, dy \, \leq \, C_6\frac{1}{\log{R}} 
\, \overset{R\to +\infty} \longrightarrow \,  0, \\
\\
&&\dys  \forall s\in(1/2,\, 1),\\
&&\qquad \qquad \frac{1}{R^{n-1}}\int_{B_R}\int_{\CC B_R}\!\!\frac{|u^*(x)-u^*(y)|^2}{|x-y|^{n+2s}}dx\, dy \, \leq \, C_6\frac{1}{R^{2s-1}} 
\, \overset{R\to +\infty} \longrightarrow \,  0,
\end{eqnarray*}
which shows the asymptotic behavior as $R$ goes to infinity of the contribution in the $H^s_0$ norm of $u^\ast$ on  $\CC B_R$, as stated in claim~(ii) and~(iii).
\vspace{1mm}

For the case $s\in (0,\, 1/2)$, the estimate in~\eqref{eq_gp3} yields
\begin{equation}\label{eq_ub}
\frac{1}{R^{n-2s}}\int_{B_R}\int_{\CC B_R}\!\!\frac{|u^*(x)-u^*(y)|^2}{|x-y|^{n+2s}}dx\, dy \, \leq \, C_6,
\end{equation}
which provides an upper bound for any $R$ large enough. Moreover, by construction of $u^\ast$, we can obtain a lower bound as follows.
\begin{eqnarray}\label{eq_lb}
\int_{B_R}\int_{\CC B_R}\!\!\frac{|u^*(x)-u^*(y)|^2}{|x-y|^{n+2s}}dx\, dy &\geq &  C_7\int_{B_{R/2}}\int_{\CC B_{2R}}\frac{dx\, dy}{|x-y|^{n+2s}}
\nonumber \\
& \geq &  C_7\int_{B_{R/2}}\!\!\!\!dx\int_{\CC B_{2R}}\frac{dy}{|y|^{n+2s}} \nonumber \\
& = &  C_8R^{n-2s},
\end{eqnarray}
for suitable positive constants $C_7$ and $C_8$, provided that $R$ is large enough. Hence, \eqref{eq_ub} together with \eqref{eq_lb} gives the estimates of the contribution in the $H^s_0$ norm of $u^\ast$ on $\CC B_R$ for the case $s\in (0,\, 1/2)$ as in claim~(i).

\vspace{2mm}
Now, notice that for any $s\in (0,1)$ using the change of variable in~\eqref{def_newvar}, $t:=\varpi x_n$, $\rho=x'/R$, we have
\begin{eqnarray}\label{1287.1}
&& \frac{1}{R^{n-1}}\!\int_{B_R}\int_{\R^n}\!\!\frac{|u^*(x)-u^*(y)|^2}{|x-y|^{n+2s}}\,dx\,dy
\nonumber \\
\nonumber \\
&& \qquad  = \frac{\varpi^{2s}}{R^{n-1}}\int_{B_R}\left[\int_{\R^{n-1}}\left(
\int_\R\frac{|u^{(0)}(\varpi x_n) -u^{(0)}(z_n)|^2}{|\varpi x_n-z_n|^{1+2s}(1+|z'|^2)^{(n+2s)/2}}\, dz_n\right)
dz'\right]dx \nonumber\\
\nonumber\\
&& \qquad = \frac{1}{\varpi}\int_{-\varpi R}^{\varpi R}
\left[ \int_{B_{\sqrt{1-|t|^2/(\varpi^2 R^2)}}}
\left(\int_\R \frac{|u^{(0)}(t)-u^{(0)}(z_n)|^2}{|t-z_n|^{1+2s}}\, dz_n\right)
d\rho \right]dt \nonumber \\
\nonumber \\
&& \qquad =\frac{\omega_{n-1}}{\varpi}\int_{-\varpi R}^{\varpi R}
\left[\int_\R \frac{|u^{(0)}(t)-u^{(0)}(z_n)|^2}{|t-z_n|^{1+2s}}\, dz_n
\right]dt -2\theta_1(R),
\end{eqnarray}
where
\begin{eqnarray}\label{def_theta1}
\theta_1(R) & := & \frac{1}{2}\frac{\omega_{n-1}}{\varpi}\int_{-\varpi R}^{\varpi R}
\Bigg[\left(1-\left(1-\frac{t^2}{\varpi^2 R^2}\right)^{n-1}\right) \nonumber \\
&& \qquad \qquad \qquad \ \int_\R \frac{|u^{(0)}(t)-u^{(0)}(z_n)|^2}{|t-z_n|^{1+2s}}\, dz_n\Bigg]dt.
\end{eqnarray}
Hence, it follows
\begin{eqnarray}\label{eq_interna}
\dys
&& \frac{1}{2}R^{1-n}\int_{B_R}\int_{B_R}\frac{|u^*(x)-u^*(y)|^2}{|x-y|^{n+2s}}\,dx\,dy \nonumber \\
\nonumber \\
&&\qquad  \quad = \frac{1}{2}R^{1-n}\left(
\int_{B_R}\int_{\R^{n}}\frac{|u^*(x)-u^*(y)|^2}{|x-y|^{n+2s}}\,dx\,dy \nonumber \right. \\
&& \left.  \qquad \qquad \qquad \qquad-\int_{B_R}\int_{\CC B_R}\frac{|u^*(x)-u^*(y)|^2}{|x-y|^{n+2s}}\,dx\,dy
\right) \nonumber \\
\nonumber \\
&&\qquad  \quad = \frac{1}{2}\frac{\omega_{n-1}}{\varpi}\int_{-\varpi R}^{\varpi R}
\left[\int_\R \frac{|u^{(0)}(t)-u^{(0)}(z_n)|^2}{|t-z_n|^{1+2s}}\, dz_n
\right]dt -\theta_2(R),
\end{eqnarray}
where
\begin{equation*}
\theta_2(R):= \frac{1}{2}R^{1-n}\int_{B_R}\int_{\CC B_R}\frac{|u^*(x)-u^*(y)|^2}{|x-y|^{n+2s}}\,dx\,dy +\theta_1(R). \nonumber
\end{equation*}
Using again the change of variable in~\eqref{def_newvar}, we have
\begin{eqnarray}\label{1287.2}
\frac{1}{R^{n-1}}\!\int_{B_R}\!W(u^*(x))\,dx & = &
\frac{\omega_{n-1}}{\varpi}\!\!
\int_{-\varpi R}^{\varpi R} W'(u^{(0)}(t))
\left(
1-\frac{t^2}{\varpi^2 R^2}
\right)^{\!\!n-1}\!\!dt \nonumber \\
\nonumber \\
&= & \frac{\omega_{n-1}}{\varpi}\!\int_{-\varpi R}^{\varpi R}\!W'(u^{(0)})(t) - \theta_3(R),
\end{eqnarray}
where
\begin{equation*}
\dys
\theta_3(R):=\frac{\omega_{n-1}}{\varpi}\!\int_{-\varpi R}^{\varpi R} W'(u^{(0)})(t)
\left(1-
\left(1-\frac{t^2}{\varpi^2 R^2}\right)^{n-1}
\right) dt.
\end{equation*}
Now we define the scaling constant $\lambda_R$ depending of $s$ as follows
\begin{equation*}
\dys
\lambda_R=
\begin{cases}
\dys \frac{1}{\log R} & \text{if} \ s = 1/2, \\
\\
1 & \text{if} \ s\in (1/2,\, 1)
\end{cases}
\end{equation*}
and we combine~\eqref{eq_interna} with~\eqref{1287.2}; we have
\begin{eqnarray}\label{eq_tutto}
&& \dys
\lambda_R R^{1-n} \FF(u^*; B_R) \nonumber \\
&& \qquad = \lambda_R R^{1-n}\cdot\Bigg\{
\frac{1}{2}\int_{B_R}\int_{B_R}\frac{|u^*(x)-u^*(y)|^2}{|x-y|^{n+2s}}\,dx\,dy \nonumber \\
&& \qquad \quad+ \int_{B_R}\int_{\CC B_R}\frac{|u^*(x)-u^*(y)|^2}{|x-y|^{n+2s}}\,dx\,dy 
+\int_{B_R} W'(u^*(x))\, dx\Bigg\} \nonumber \\
\nonumber \\
&& \qquad = \lambda_R\cdot\Bigg\{\,
\frac{1}{2}\frac{\omega_{n-1}}{\varpi}\int_{-\varpi R}^{\varpi R}
\left[\int_{\R}\frac{|u^{(0)}(t)-u^{(0)}(z_n)|^2}{|t-z_n|^{1+2s}}\,dz_n\right]dt  \\
&& \qquad \qquad \qquad  +\frac{1}{2}R^{1-n}\int_{B_R}\int_{\CC B_R} \frac{|u^*(x)-u^*(y)|^2}{|x-y|^{n+2s}}\,dx\, dy \nonumber \\
&& \qquad \qquad  \qquad + \frac{\omega_{n-1}}{\varpi}\int_{-\varpi R}^{\varpi R}W'(u^{(0)}(t))\, dt  -\theta_4(R) \,\Bigg\}, \nonumber
\end{eqnarray}
where
$$
\dys \theta_4(R) = \big( \theta_2(R)+\theta_3(R) \big).
$$
\vspace{1mm}

We observe that
\begin{equation}\label{eq_resto}
\dys \liminf_{R\rightarrow+\infty} \lambda_R \,\theta_4(R) = 0.
\end{equation}
Indeed, recalling that~${\mathcal{G}}(u^{(0)})$ is finite, due
to Theorem~\ref{theorem_unod},
it suffices to recall~\eqref{eq_gp3}
and apply~Lemma~\ref{lem_aux} with
$$
\alpha(t,R)=\frac{\omega_{n-1}}{\varpi}\left(1-\left(1-\frac{|t|^2}{R^2}\right)^{n-1}\right)
$$
and
$$
\dys \beta(t)=\frac{1}{2}\int_{\R}\frac{|u^{(0)}(t)-u^{(0)}(z_n)|^2}{|t-z_n|^{1+2s}}\, dz_n
+ W'(u^{(0)})(t).
$$
\vspace{1mm}

Thus, we make use of~\eqref{eq_gp3} and~\eqref{eq_resto}, so that by taking the limit as $R\to+\infty$ in~\eqref{eq_tutto} we obtain
\begin{eqnarray*}
 \liminf_{R\rightarrow+\infty}
\lambda_R R^{1-n} \FF(u^*; B_R) \, =\, \frac{\omega_{n-1}}{\varpi}
{\mathcal{G}} (u^{(0)}).
\end{eqnarray*}
This completes the proof of claim~(ii) and (iii).
\vspace{2mm}

Finally, using Lemma~\ref{lip_lem-0} with~$\rho:=1$, we obtain
$$ \int_{\R^n} \frac{|u^*(x)-u^*(y)|^2}{|x-y|^{n+2s}}\,dy\le C_9,$$
for any~$x\in\R^n$, for a suitable~$C_9>0$, and so
$$ \FF(u^*; B_R\setminus B_{(1-\delta)R}) \le
\Big( C_9 +\sup_{r\in[-1,1]} W(r)\Big) \,\big|B_R\setminus 
B_{(1-\delta)R}\big|,$$
that is~\eqref{1-del}.
The proof of the theorem is complete.\qed
\end{proof}

\vspace{2mm}

\section{Appendix}
In this Appendix we state and prove some general results involving the 
Ga\-gliar\-do norm $\|\cdot\|_{H^s}$ and various results, that are 
necessary for the proofs of the main results of this papers.
\vspace{1mm}

As usual in this paper, throughout this section we will assume that the fractional exponent $s$ is a real number belonging to $(0,1)$.

\subsection{Regularity properties of the fractional Allen-Cahn equation}
The following propositions recall how the fractional Laplacian operators interact with the $C^{\alpha}$-norms. Their proofs can be found in~\cite[Chapter 2]{silvestre}, which presents some general properties of the $(-\Delta)^s$ operators and provides characterization of its supersolutions (see also \cite{silvestre07} and \cite{CS10}).
\vspace{1mm}

\begin{proposition}{\rm (}\cite[Proposition 2.1.10]{silvestre}{\rm)}\label{pro_sil}
Let $n\ge 1$. Let $w\in C^{0,\alpha}(\R^n)$, for $\alpha\in (0,1]$.
Let $u \in L^{\infty}(\R^n)$ be such that
\begin{equation}\label{eq_silv}
-(-\Delta)^s u(x) = w(x) \ \ \text{for any} \ x \in \R^n.
\end{equation}
Then,
\begin{itemize}
\item[(i)]{
If $\alpha+2s\leq 1$, then $u \in C^{0,\alpha+2s}(\R^n)$. Moreover
$$
\|u\|_{C^{0,\alpha+2s}(\R^n)} \le C\big(\|u\|_{L^{\infty}(\R^N)}+\|w\|_{C^{0,\alpha}(\R^n)}\big)
$$
for a constant $C$ depending only on $n, \alpha$ and $s$.
}\vspace{1mm}
\item[(ii)]{
If $\alpha+2s>1$, then $u \in C^{1,\alpha+2s-1}(\R^n)$. Moreover
$$
\|u\|_{C^{1,\alpha+2s-1}(\R^n)} \le C\big(\|u\|_{L^{\infty}(\R^n)}+\|w\|_{C^{0,\alpha}(\R^n)}\big)
$$
for a constant $C$ depending only on $n, \alpha$ and $s$.
}
\end{itemize}\vspace{1mm}
\end{proposition}

\begin{proposition}{\rm (}\cite[Proposition 2.1.11]{silvestre}{\rm)}\label{pro_sil2}
Let $n\ge 1$. Let $u$ and $w \in L^{\infty}(\R^n)$ be such that
\begin{equation*}
-(-\Delta)^s u(x) = w(x) \ \ \text{for any} \ x \in \R^n.
\end{equation*}
Then,
\begin{itemize}
\item[(i)]{
If $2s\leq 1$, then $u \in C^{0,\alpha}(\R^n)$ for any $\alpha<2s$. Moreover
$$
\|u\|_{C^{0,\alpha}(\R^n)} \le C\big(\|u\|_{L^{\infty}(\R^n)}+\|w\|_{L^{\infty}(\R^n)}\big)
$$
for a constant $C$ depending only on $n, \alpha$ and $s$.
}\vspace{1mm}
\item[(ii)]{
If $2s>1$, then $u \in C^{1,\alpha}(\R^n)$ for any $\alpha<2s-1$. Moreover
$$
\|u\|_{C^{1,\alpha}(\R^n)} \le C\big(\|u\|_{L^{\infty}(\R^n)}+\|w\|_{L^{\infty}(\R^n)}\big)
$$
for a constant $C$ depending only on $n, \alpha$ and $s$.
}
\end{itemize}\end{proposition}

We remark that the above results (and, consequently, the claims in the forthcoming Lemma~\ref{lem_co}) are valid also for solutions of~\eqref{eq_silv} in bounded domains, leading to a local regularity theory.

\vspace{2mm}

Since we deal with the case of $w$ in~\eqref{eq_silv} being the derivative of a double-well potential $W$,  
we have to extrapolate the regularity informations for the solutions of equation (\ref{equazione}); this can be obtained by iterating the results in Proposition \ref{pro_sil} and Proposition~\ref{pro_sil2}.
In the following two lemmas we arrange some regularity results in the form to be applied in this paper (as well as in~\cite{sv0} and~\cite{sv1}).

\begin{lemma}\label{lem_co}
Let $n\ge 1$. Let $u\in L^{\infty}(\R^n)$ be such that
\begin{equation}\label{eq_giusta}
-(-\Delta)^s u(x) = W'(u(x)) \ \ \text{for any} \ x \in \R^n,
\end{equation}
with $W\in C^1(\R)$.
Then,
\begin{itemize}
\item[(i)]{
If $s \in (0,\, 1/2]$, then $u \in C^{0,\alpha}(\R^n)$ for any $\alpha<2s$. Moreover, 
$$
\|u\|_{C^{0,\alpha}(\R^n)} \le C\big(\|u\|_{L^{\infty}(\R^n)}+\|W'(u)\|_{L^{\infty}(\R^n)}\big).
$$
}
\item[(ii)]{
If $s\in (1/2,\, 1)$, then $u \in C^{1,\alpha}(\R^n)$  for any $\alpha<2s-1$. Moreover, 
$$
\|u\|_{C^{1,\alpha}(\R^n)} \le C\big(\|u\|_{L^{\infty}(\R^n)}+\|W'(u)\|_{L^{\infty}(\R^n)}\big),
$$
}
\end{itemize}
for a constant $C$ depending only on $n, \alpha$ and $s$.
\end{lemma}

\begin{proof}
The proof is immediate.
Let $u$ in $L^{\infty}(\R^n)$ 
be a solution of equation (\ref{eq_giusta}). Since $W$ belongs to $C^1(\R)$, it suffices to apply Proposition \ref{pro_sil2}(i)-(ii) by chosing $w(x):=W'(u(x))$.\qed
\end{proof}

\vspace{1mm}

\begin{lemma}\label{lem_dif}
Let $n\geq1$ and let $u\in L^{\infty}(\R^n)$ satisfy equation~\eqref{eq_giusta}, with $W\in C^2({\R})$.
Then $u\in C^{2,\alpha}(\R^n)$, with $\alpha$ depending on $s$.
\end{lemma}
\begin{proof}
Let $s\in (1/2,\, 1)$ and let $u$ in $L^{\infty}(\R^n)$ be a solution of the equation~(\ref{eq_giusta}). Then, $u\in C^{1,\alpha}(\R^n)$ with its $C^{1,\alpha}$ norm bounded as in Lemma~\ref{lem_co}(i). Moreover $u'$ satisfies
\begin{equation}\label{eq_giustader}
-(-\Delta)^su'(x)=W''(u(x))u'(x) \ \ \text{for any} \ x\in \R^n.
\end{equation}
By the hypothesis on $W$ 
and $u$, we can apply Proposition~\ref{pro_sil2}(ii) to the solution $u'$ of equation (\ref{eq_giustader}) with $w:=W''(u(x))u'(x)$. It follows that $u'$ belongs to $C^{1,\alpha}(\R^n)$ for any $\alpha<2s-1$ and thus the claim is proved.
\vspace{1mm}

Let $s =1/2$. Then, by the fact that $W$ is in $C^2$ together with the regularity of $u$ provided by Lemma~\ref{lem_co}(i), Proposition~\ref{pro_sil}(ii) with $w:=W'(u)$ yields that the function $u$ belongs to $C^{1, \alpha}(\R^n)$ for any $\alpha<1$. Now, we can argue as for the case $s\in(1/2,\, 1)$ to obtain the desired regularity for $u$ by Proposition~\ref{pro_sil2}(ii).
\vspace{1mm}

Finally, let $s\in (0,\, 1/2)$ and let $u \in L^{\infty}(\R^n)$ 
be a solution of~\eqref{eq_giusta}. So,
Lemma~\ref{lem_co}(i) yields 
$u\in 
C^{0,\alpha}(\R^n)$ for any $\alpha<2s$. Then, for $s\in (1/4,\,1/2)$ we can apply Proposition~\ref{pro_sil}(ii) and we get $u \in C^{1,\alpha+2s-1}(\R^n)$. Hence, $u'$ is well defined and it satisfies equation~\eqref{eq_giustader} with $w=W''(u)u'$ belonging to $C^{0,\alpha+2s-1}(\R^n)$ and again by Proposition~\ref{pro_sil}(ii) we get $u' \in C^{1,\alpha+2s-1}$ for any $\alpha<2s$.

For $s \in (0,\, 1/4]$, we can use Proposition~\ref{pro_sil}(i) in order to obtain $u \in C^{0, \alpha+2s}(\R^n)$ for any $\alpha <2s$. Thus, when $s\in (1/6,\, 1/4)]$, we can apply twice Proposition~\ref{pro_sil}(ii) arguing as in the case $s\in (1/4,\,1/2)$ and we get $u' \in C^{1,\alpha+4s-1}(\R^n)$, for any $\alpha<2s$.

By iterating the above procedure on $k\in \N$, we obtain that, when $s \in (1/(2k+2),\,1/2k]$, $u$ belongs to $C^{2,\alpha+2k-1}$ for any $\alpha<2s$.\qed
\end{proof}

\vspace{2mm}

We conclude this section observing that the equation we deal with
behaves well under limits:

\begin{lemma}\label{lem_limiti} Let $W\in C^1(\R)$.
For any~$k\in \N$, let~$u_k\in C(\R^n)\cap L^\infty(\R^n)$ be such that
$$
-(-\Delta)^s u_k(x) = W'(u_k(x)) \ \  \text{for any} \ x \in B_k.
$$
Suppose that~$\dys \sup_{k}\|u_k\|_{L^{\infty}(\R^n)}<\infty$ and that $u_k$ converges a.e. to a function~$u$.
Then,
$$ -(-\Delta)^s u(x) = W'(u(x)) \ \ \text{for any} \ x \in\R^n.
$$
\end{lemma}

\begin{proof}
Given any~$\phi\in C^\infty_0(\R)$ supported in $B_k$,
\begin{eqnarray*}\int_\R
W'(u_k(x))\,\phi(x)\,dx \! &=&\! \int_\R\left[
\int_{\R}\frac{u_k(y)-u_k(x)}{|x-y|^{n+2s}}\,dy\right]\,\phi(x)\,dx
\\
\\ &=& \! \int_{\R}\int_{\R}
\frac{u_k (x)\,\big(\phi(y)-\phi(x)\big)}{|x-y|^{n+2s}}\,dx\,dy.
\end{eqnarray*}
\vspace{1mm}

Moreover,
\begin{eqnarray*}
&&\int_{\R} \Big| \int_{\R}\frac{\phi(x)-\phi(y)}{|x-y|^{n+2s}}\,dy \Big| dx \ =  \ \int_{\R} \Big| \int_{\R}\frac{|\phi(x)-\phi(x-y)|}{|y|^{n+2s}}\,dy \Big| dx \\
\\
&& \quad \leq \ \int_{\R}\!dx \left[ \Big| \int_{B_1}\frac{\phi(x)-\phi(x+y)+\nabla\phi(x)y}{|y|^{n+2s}}\,dy \Big| + \Big| \int_{\CC B_1}\frac{2\|\phi\|_{L^\infty}}{|y|^{n+2s}}\,dy \Big| \right] dx \\
\\
&& \quad \ \leq  \int_{\R} \!dx \Big| \int_0^1 \frac{\|\nabla^2\phi\|_{L^{\infty}}}{r^{n+2s}}r^{n+1}\,dr 
+ \int_1^{\infty}\frac{2\|\phi\|_{L^{\infty}}}{r^{1+2s}}\,dr \Big| \ < \ +\infty.
\end{eqnarray*}
\vspace{1mm}

Thus, by
Dominated Convergence Theorem,
\begin{equation*}
\begin{split}
\int_\R W'(u (x))\phi(x)\,dx\,
\! &= \int_{\R}\int_{\R}
\frac{u (x)\,\big(\phi(y)-\phi(x)\big)}{|x-y|^{n+2s}}\,dx\,dy
\\
\\ &= \int_\R\left[
\int_{\R}\frac{u (y)-u (x)}{|x-y|^{n+2s}}\,dy\right]\,\phi(x)\,dx
\\
\\&= \int_\R -(-\Delta)^s u (x)\phi(x)\,dx,
\end{split}
\end{equation*}
which gives the desired claim, since~$\phi$ is arbitrary.\qed
\end{proof}

\vspace{2mm}

\subsection{Construction of barriers}\label{sec_preliminare}

We start by recalling the construction
of an useful barrier,
that is used in \cite{sv2011,sv0,sv1} and also here in the asymptotic analysis of the one-dimensional minimizers of the energy~\eqref{energia} (see the forthcoming Proposition~\ref{pro_control}). The proof can be found in~\cite[Lemma 3.1]{sv0}; it relies on a fine construction around the power function $t\mapsto|t|^{-2s}$ together with some estimates proved here in the following.

\begin{lemma}\label{tau}
{\rm(\cite{sv0})}. Let~$n\ge1$. 
Given any~$\tau>0$, there exists
a constant $C> 1$, possibly
depending on~$n$, $s$ and $\tau$,
such that the following holds: for any $R\ge
C$, there exists a rotationally symmetric function
\begin{equation}\label{sopra}
w \in C\big( \R^n ; [-1+C R^{-2s},\,1]\big),
\end{equation}
with
\begin{equation}\label{sopra2}{\mbox{
$w=1$ in~$\CC B_R$,}}\end{equation} such that
\begin{equation}\label{al1}
\int_{\R^n}\frac{w(y)-w(x)}{{|x-y|^{n+2s}}}\,dy
\le \tau\big(1+w(x)\big)\end{equation}
and
\begin{equation}\label{al2}
\frac{1}{C}(R+1-|x|)^{-2s} \leq 1+w(x)\le C \big( R+1-|x|\big)^{-2s}
\end{equation}
for any~$x\in B_R$.
\end{lemma}

\vspace{2mm}

Now, we consider the following equation related to the fractional operator $(-\Delta)^s$ on the real line,
\begin{equation}\label{eq_auto}
-(-\Delta)^sv(x) - \alpha v(x)=0,
\end{equation}
where $\alpha$ is a positive constant.
Precisely, in Corollary~\ref{cor_super} we show that the function $v$ being a subsolution of equation~\eqref{eq_auto} away from the origin is bounded (up to a multiplicative constant) by the function $x\mapsto|x|^{-(1+2s)}$. This estimate will be crucial in the analysis of the global minimizers of the functionals $\FF$  
(see Theorem~\ref{theorem_unod}).

\vspace{1mm}

First, we need to prove the following 1-D result.

\begin{lemma}\label{sop}
Let~$\eta\in C^2(\R;(0,+\infty))$, with~$\| \eta\|_{C^2(\R)}<+\infty$, and
$$ \eta(x)=\frac{1}{|x|^{1+2s}}\qquad{\mbox{ for any }}x\in\R\setminus
(-1,1).$$
Then there exists~$\kappa\in(0,+\infty)$, possibly depending
on~$s$ and~$\eta$, such that
$$ \limsup_{x\rightarrow\pm\infty}
\frac{
-(-\Delta)^s \eta(x)
}{\eta(x)}\le\kappa.$$
\end{lemma}

\begin{proof} We will denote by~$C$ suitable positive
quantities, possibly different from line to line,
and possibly depending on~$s$ and~$\eta$.
For all~$(x,y)\in\R^2$ with~$|x|\ge2$, we define
$$ i(x,y):=
\frac{
\eta(y)-\eta(x)-\chi_{(-1/4,1/4)}(x-y)\,\eta'(x)(y-x)}{
|x-y|^{1+2s}}.$$
For any fixed~$y\in\R$, we have that
\begin{equation}\label{BB6}\begin{split}&
\lim_{x\rightarrow\pm\infty}
|x|^{1+2s}i(x,y)
=\lim_{x\rightarrow\pm\infty}
\frac{|x|^{1+2s}}{|x-y|^{1+2s}}
\big( \eta(y)-\eta(x)\big)=\eta(y).
\end{split}\end{equation}
Also,
if~$|y|\le1$ and~$|x|\ge 2$, we have that~$|x-y|\ge|x|-|y|\ge|x|/2$
and so
\begin{equation}\label{BB5}
\begin{split}&
|x|^{1+2s} |i(x,y)|
=\frac{|x|^{1+2s}
\big|\eta(y)-\eta(x)\big|}{|x-y|^{1+2s}}
\le \dys 16\sup_{\R}|\eta|.\end{split}
\end{equation}
Using~\eqref{BB6}, \eqref{BB5}
and the Bounded Convergence Theorem, we conclude that
\begin{equation}\label{C18}\begin{split}&
\lim_{x\rightarrow\pm\infty} |x|^{1+2s}
\int_{-1}^{1}
\frac{
\eta(y)-\eta(x)-\chi_{(-1/4,1/4)}(x-y)\,\eta'(x)(y-x)}{     
|x-y|^{1+2s}}\,dy\\
& \qquad
=\int_{-1}^{1}\lim_{x\rightarrow\pm\infty} |x|^{1+2s}  
i(x,y)\,dy
=\int_{-1}^1 \eta(y)\,dy.
\end{split}\end{equation}
Now, fixed~$|x|\ge2$,
we estimate the contribution in~$\R\setminus(-1,1)$.
We write~$\R\setminus(-1,1)=P\cup Q\cup R\cup S$, where
\begin{eqnarray*}
P &=& \Big\{ y\in \R\setminus(-1,1) {\mbox{ s.t. }} |x|/2< |y|\le 2|x|
{\mbox{ and }} |x-y|\ge1/4\Big\},\\
Q &=& \Big\{ y\in \R\setminus(-1,1) {\mbox{ s.t. }} |x|/2< |y|\le 2|x|  
{\mbox{ and }} |x-y|< 1/4\Big\},\\
R &=& \Big\{ y\in \R\setminus(-1,1) {\mbox{ s.t. }} |y|> 2|x|  
\Big\},\\
S &=& \Big\{ y\in \R\setminus(-1,1) {\mbox{ s.t. }} |y|\le |x|/2  
\Big\}.
\end{eqnarray*}
We observe that, if~$y\in P$,
\begin{eqnarray}\label{stima_p}
|i(x,y)| & = &
\frac{|
\eta(y)-\eta(x)|}{
|x-y|^{1+2s}}\, \le \,
\frac{|
\eta(y)|+|\eta(x)|}{
|x-y|^{1+2s}} 
 \,
 = \, \frac{(1/|y|^{1+2s})
+(1/|x|^{1+2s})}{
|x-y|^{1+2s}}
\nonumber \\
& \le & \frac{C}{|x|^{1+2s}
|x-y|^{1+2s}}.
\end{eqnarray}
As a consequence,
\begin{eqnarray*}
|x|^{1+2s}\int_P i(x,y)\,dy \le C
\int_P \frac{dy}{|x-y|^{1+2s}}
\le C
\int_{\{|x-y|\ge1/4\}} \frac{dy}{|x-y|^{1+2s}} 
\le C.\end{eqnarray*}
Moreover, if~$y\in Q$, we can use the Taylor expansion
of the function~$1/|t|^{1+2s}$ to obtain that
\begin{eqnarray*}
&& \eta(y)-\eta(x)-\chi_{(-1/4,1/4)}(x-y)\,\eta'(x)(y-x)\\
&&  \qquad \qquad \qquad \qquad  \qquad
= \, \eta(y)-\eta(x)-\eta'(x)\cdot(y-x)
\\ &&\qquad \qquad \qquad \qquad  \qquad
= \, \frac{1}{|y|^{1+2s}}-\frac{1}{|x|^{1+2s}}
+\frac{(1+2s)}{|x|^{3+2s}}x(y-x)
\\ && \qquad \qquad \qquad \qquad  \qquad
= \, \frac{(1+2s)(2+2s)}{|\xi|^{3+2s}}|x-y|^2,
\end{eqnarray*}
for an appropriate $\xi$ which lies on the
segment joining~$x$ to~$y$. Notice also that if~$y\in Q$,
then~$y\ge0$ if and only if~$x\ge0$, therefore both~$x$
and~$y$ lie
either in~$[|x|/2,+\infty)$ or in~$(-\infty, -|x|/2]$.
In any case,~$|\xi|\ge |x|/2$ and so,
for any~$y\in Q$,
\begin{eqnarray*}
|i(x,y)| & \!= \! &
\frac{|\eta(y)-\eta(x)-\chi_{(-1/4,1/4)}(x-y)\,\eta'(x)(y-x)|}{|x-y|^{1+2s}}\\
\\ &= & \frac{C}{|\xi|^{3+2s}} |x-y|^{1-2s}
\ \le \ \frac{C}{|x|^{3+2s}} |x-y|^{1-2s} .
\end{eqnarray*}
As a consequence,
\begin{eqnarray*}
|x|^{1+2s}\int_Q i(x,y)\,dy
&\le&  \frac{C}{|x|^2}
\int_Q {|x-y|^{1-2s}}
\\
&\le & \frac{C}{|x|^2}
\int_{|x-y|<1/4} {|x-y|^{1-2s}}
\ \le \ \frac{C}{|x|^2} \ \le \ C.
\end{eqnarray*}
Furthermore, if~$y\in R$, we have that~$|x-y|\ge |y|-|x|\ge |x|>1/4$,
 thus we can estimate the function $i(x,y)$ as in~\eqref{stima_p} and we obtain
\begin{eqnarray*}
|i(x,y)| \, \le \, \frac{C}{|x|^{1+2s}|x-y|^{1+2s}}.
\end{eqnarray*}
In particular,
\begin{eqnarray*}
|x|^{1+2s}\int_R i(x,y)\,dy
 & \le &  C\!
\int_{\{|y|\ge2|x|\}} \frac{dy}{|x-y|^{1+2s}}
\\
&\le &  C\!
\int_{\{|x-y|\ge|x|\}} \frac{dy}{|x-y|^{1+2s}}
\ = \ \frac{C}{|x|^{2s}} \ \le \ C.
\end{eqnarray*}
As for the last contribution, if~$y\in S$
then~$|x-y|\ge|x|-|y|\ge|x|/2\ge 1$ and so
\begin{eqnarray*} 
|i(x,y)| \! & \le & \! \frac{|\eta(y)|+|\eta(x)|}{|x-y|^{1+2s}}
 =  \frac{(1/|y|^{1+2s})+(1/|x|^{1+2s})
}{|x-y|^{1+2s}} \\
& \le & 
\frac{C}{|x|^{1+2s}|y|^{1+2s}}
.\end{eqnarray*}
Accordingly,
\begin{eqnarray*}
&& |x|^{1+2s}\int_S i(x,y)\,dy \, \le \, C\!
\int_{\{
1\le|y|\le |x|/2\}
} \frac{dy}{|y|^{1+2s}}
\, \le \, C.
\end{eqnarray*}
All in all, we obtain that
\begin{eqnarray*}
&& \limsup_{x\rightarrow\pm\infty} |x|^{1+2s}
\int_{\R\setminus(-1,1)}
\frac{     
\eta(y)-\eta(x)-\chi_{(-1/4,1/4)}(x-y)\,\nabla\eta(x)\cdot(y-x)}{
|x-y|^{1+2s}}\,dy
\\
&&\qquad=
\limsup_{x\rightarrow\pm\infty} |x|^{1+2s}\left(
\int_{P}i(x,y)\,dy+
\int_{Q}i(x,y)\,dy\right.\\
&&\qquad\quad\left.+\int_{R}i(x,y)\,dy+\int_{S}i(x,y)\,dy\right)
\ \le  \ C.
\end{eqnarray*}
{F}rom this and~\eqref{C18}, the desired result plainly
follows.\qed
\end{proof}

\begin{corollary}\label{cor_super}
Let~$\alpha$, $\beta>0$. 
Let~$v$ be a bounded function in $C^{0,\gamma}(\R)$, with $\gamma>2s$,
such that~$-(-\Delta)^s v(x)\ge\alpha v(x)$
for any~$x\in\R\setminus(-\beta,\beta)$.
Then, there exists a constant~$\bar{C}>0$, possibly depending on~$s$,
$\alpha$ and~$\beta$, such that
$$ v(x)\leq\frac{\bar{C}}{|x|^{1+2s}}\qquad{\mbox{
for any }}x\in\R.$$
\end{corollary}

\begin{proof}
If $v$ is identically 0, we have done. So, we suppose $\|v\|_{L^{\infty}(\R)}>0$.

Take~$\eta$ and~$\kappa$
as in Lemma~\ref{sop}.
Define
$$ a:=\left( \frac{\alpha}{2\kappa}\right)^{1/(2s)}$$
and~$\zeta(x):=\eta(ax)$.

Then,
\begin{eqnarray*}
\limsup_{x\rightarrow\pm\infty}
\frac{
-(-\Delta)^s \zeta(x)
}{\zeta(x)}=a^{2s}
\limsup_{x\rightarrow\pm\infty}
\frac{ 
-(-\Delta)^s \eta(ax)
}{\eta(ax)}
\le a^{2s}\kappa
=\frac\alpha2.
\end{eqnarray*}
As a consequence, there exists~$\beta'\ge\beta$ such that
\begin{equation}\label{9067}
{\mbox{$-(-\Delta)^s \zeta(x)\le\alpha\zeta(x)$ for
any~$x\in\R\setminus(-\beta',\beta')$.}}
\end{equation}
Now, we set
$$ \bar{C}:=\frac{4\| 
v\|_{L^\infty(\R)}}{\displaystyle\min_{[-a\beta',a\beta']}\eta}
=\frac{4\| v\|_{L^\infty(\R)}}{
\displaystyle\min_{[-\beta',\beta']}\zeta}.$$
We claim that
\begin{equation}\label{eq_control}
{\mbox{$v(x)\le \bar{C}\zeta(x)$ for any $x\in\R$.}}
\end{equation}
In order to prove the above inequality,
we take~$b$ in $[0, +\infty)$ and we define~$v_b(x):= \bar{C}\zeta(x)+b-v(x)$.
When~$b>\| v\|_{L^\infty(\R)}$, we have that~$v_b(x)>0$
for any~$x\in\R$. Now,
if $v_b(x)>0$ $\forall x\in \R$ and $\forall b \in [0, +\infty]$, we take $b:=0$ and we get~\eqref{eq_control}. Then, we may take~$b_o$ the first~$b$ for which~$v_b$
touches~$0$ from above: we have that~$v_{b_o}(x)\ge 0$ 
and that there exists a sequence~$x_k\in\R$ such that~$v_{b_o}(x_k)
\le 2^{-k}$, for~$k\in\N$.
We claim that
\begin{equation}\label{b0}
b_o = 0.
\end{equation}
Indeed, we have, if~$k$ is sufficiently large, 
\begin{equation*}
\| v\|_{L^\infty(\R)}\ge
2^{-k} \ge v_{b_o}(x_k)\geq \bar{C}\zeta(x_k)-v(x_k)\ge
\bar{C}\zeta(x_k) -\| v\|_{L^\infty(\R)}
\end{equation*}
and so
$$ \zeta(x_k) \le 
\frac{ 2\| v\|_{L^\infty(\R)} }{\bar{C}}
=\frac{{
\displaystyle\min_{[-\beta',\beta']}\zeta}}{2}.$$
Therefore,~$|x_k|>\beta'$.

Hence, recalling~\eqref{9067},
\begin{equation}\label{189}
\begin{split}
&\int_\R \frac{v_{b_o}(y)-v_{b_o}(x_k)}{|x_k-y|^{1+2s}}\,dy
= -(-\Delta)^s v_{b_o}(x_k) \\ &\qquad=
-\bar{C} (-\Delta)^s\zeta(x_k)
+(-\Delta)^s v(x_k)
\, \le \, \alpha (\bar{C}\zeta(x_k) -v(x_k))
\\ &\qquad=\alpha v_{b_o}(x_k)-\alpha b_o
\,\le \,
2^{-k}\alpha -\alpha b_o.
\end{split}\end{equation}
Now, we define~$v_k(x):=v_{b_o}(x+x_k)$.
Notice that~$v_k(x)\ge 0$ for any~$x\in\R^n$ and~$v_k(0)\le 2^{-k}$.
Also, by the Theorem of Ascoli, up to subsequence, we may
suppose that~$v_k$ converges to some~$v_\infty$ 
locally uniformly as $k\rightarrow
+\infty$.
It follows that
\begin{equation}\label{eq_vpos}
v_\infty(x)\ge0 \, =\, v_\infty(0) \ \ \ \text{for any} \ x\in\R.
\end{equation}
Moreover, by the assumption on the function $v$, we obtain that $v_k$ satisfies
$$
\dys \frac{|v_k(t)-v_k(0)|}{|t|^{1+2s}}\, \leq \, C\big(|t|^{\gamma-1-2s}\chi_{(-1,1)}(t) + |t|^{-1-2s}\chi_{\R\setminus(-1,1)}(t)\big),
$$
for a suitable constant $C>0$. Thus, the Dominated Convergence Theorem yields
\begin{equation}\label{eq_vlimit}
\lim_{k\rightarrow+\infty}
\int_\R \frac{v_{k}(t)-v_{k}(0)}{|t|^{1+2s}}\,dt \, = \, \int_\R \frac{v_{\infty}(t)}{|t|^{1+2s}}\,dt.
\end{equation}
Finally, combining the above equation with ~\eqref{189} and~\eqref{eq_vpos}, we get 
\begin{eqnarray*}
-\alpha b_o &\ge& \lim_{k\rightarrow+\infty}
\int_\R \frac{v_{b_o}(y)-v_{b_o}(x_k)}{|x_k-y|^{1+2s}}\,dy
\\ &=&
\lim_{k\rightarrow+\infty}
\int_\R \frac{v_{k}(t)-v_{k}(0)}{|t|^{1+2s}}\,dt
\\ &=& 
\int_\R \frac{v_{\infty}(t)}{|t|^{1+2s}}\,dt \ \geq \ 0.
\end{eqnarray*}

This completes the proof of~\eqref{b0}.
\vspace{1mm}

Now, from~\eqref{b0}, we conclude that, for any~$x\in\R$,
$$ 0\le v_{b_o}(x)=
\bar{C}\zeta(x)+b_o-v(x)=
\bar{C}\zeta(x)-v(x)$$
and so~$v(x)\le \bar{C}\zeta(x)$.\qed
\end{proof}

\vspace{2mm}

We finish this section by using the barriers constructed in Lemma~\ref{tau} and Lemma~\ref{sop} in order to obtain a precise control on the behavior at infinity of the monotone solutions of equation~\eqref{equazione}. 

\begin{proposition}\label{pro_control}
Let $n=1$ and let $W\in C^2(\R)$ be a double-well potential with wells at $\{-1,+1\}$ such that $W''(\pm1)>0$. Suppose that $u$ is a strictly increasing function which satisfies
\begin{equation}\label{eqqq}
\begin{cases}
-(-\Delta)^su(x) = W'(u(x)) \ \text{for any} \  x\in \R, \\
\dys \lim_{x\to\pm\infty}u(x) = \pm 1.
\end{cases}
\end{equation}
Then there exists a constant $C\geq 1$ such that
\begin{equation}\label{EST5}
|{u} (x)-\,{\rm sign}\,(x)|\le C\,|x|^{-2s},\\
\end{equation}
\begin{equation}\label{stima_der}
\big|{u}' (x)\big|\le C\,|x|^{-(1+2s)}
\end{equation}
for any large $x\in\R$.
\end{proposition}
\begin{proof}
First, we note that the potential $W$ satisfies
\begin{equation}\label{grow3}
{\mbox{$W'(t)\ge W'(r)+c(t-r)$ when
$r\le t$, $r,t\in [-1,\,-1+c]\cup[+1-c,\,+1]$,
}}
\end{equation}
for some~$c>0$.

Now, we choose $\tau=c$ in Lemma~\ref{tau} and, for any~$R\ge C$, we consider the barrier $w$ constructed there.

{F}rom~\eqref{sopra}, we know that there exists~$K\in\R$ such that, 
if~$k\in (-\infty,K]$, then~$w(x-k)> u(x)$ for any~$x\in\R$.
We take~$\bar{k}$ as large as possible with this property, i.e.,
\begin{equation}\label{7889d12}
{\mbox{$w(x-k)> u(x)$ for any~$k<\bar{k}$ and
any $x\in\R$}}\end{equation} 
and
there exists an infinitesimal sequence~$\eta_j\in[0,1)$ and 
points~$x_j\in\R$ for
which \begin{equation}\label{dhuju2112}
w(x_j-(\bar{k}+\eta_j))\le 
u(x_j).\end{equation}
{F}rom the asymptotic behavior at $\infty$ and the strict monotonicity of~$u$,
we know that~$|u(x)|<1$ for any~$x\in\R$.
Hence, by~\eqref{dhuju2112},
$$ w(x_j-(\bar{k}+\eta_j))< 1.$$
This and~\eqref{sopra2} gives that
\begin{equation}\label{909878124556}
\big| x_j-(\bar{k}+\eta_j) \big|\le R,\end{equation}
therefore
$$ |x_j|\le R+|\bar{k}|+1.$$

Thus, up to subsequence, we may suppose that
$$ \lim_{j\rightarrow+\infty} x_j=\bar{x},$$
for some~$\bar{x}\in\R$. Moreover,~\eqref{909878124556}
implies that
\begin{equation}\label{909878124556-2}
\bar{x}-\bar{k}\in [-R,R],\end{equation}
while~\eqref{dhuju2112} 
and~\eqref{7889d12}
give that~$w(\bar{x}-\bar{k})=u(\bar{x})$.
\vspace{1mm}

Thus, we set~$v(x):=w(x-\bar{k})-u (x)$ and we see that~$v(x)\ge0$
for any~$x\in\R$ and~$v(\bar{x})=0$.
\vspace{1mm}

Note that if~$x-\bar{k}\in [-R,R]$ and $u(x)\in [-1,-1+c]$, then
\begin{eqnarray}\label{909878124556-3} 
\int_\R \frac{v(y)-v(x)}{|x-y|^{1+2s}}\,dy  & = &
\int_\R \frac{w(y-\bar{k})-w(x-\bar{k})}{|x-y|^{1+2s}}\,dy
+(-\Delta)^s u(x) \nonumber \\
&\le &  \tau (1+w(x-\bar{k}))
-W'(u(x)) \nonumber\\
& \le &  \tau(1+w(x-\bar{k}))-c(u(x)+1) \nonumber \\
 &= &  c v(x),
\end{eqnarray}
thanks to~~\eqref{al1},~\eqref{eqqq} and~\eqref{grow3}.
\vspace{1mm}

We claim that
\begin{equation}\label{87dfhjh3nd72h2}
u(\bar{x})>-1+c.
\end{equation}
The proof of~\eqref{87dfhjh3nd72h2}
is by contradiction: if~$u(\bar{x})\in [-1,\,-1+c]$
we deduce from~\eqref{909878124556-2}
and~\eqref{909878124556-3}
that
$$ 
\int_\R \frac{v(y)}{|\bar{x}-y|^{1+2s}}\,dy
\,=\,
\int_\R \frac{v(y)-v(\bar{x})}{|\bar{x}-y|^{1+2s}}\,dy
\, \le \,cv(\bar{x})\,=\,0.$$
Since the first integrand is non-negative, we would have that~$v$
vanishes identically, i.e.~$w(x-\bar{k})=u (x)$ for any~$x\in\R^n$.
But then
$$ +1=\lim_{x\rightarrow-\infty}w(x-\bar{k})=
\lim_{x\rightarrow-\infty}u (x)=-1$$
and this contradiction proves~\eqref{87dfhjh3nd72h2}.
\vspace{1mm}

{F}rom~\eqref{al2}, \eqref{909878124556-2}
and~\eqref{87dfhjh3nd72h2}, we 
obtain
$$ C \big( R+1-|\bar{x}-\bar{k}|\big)^{-2s}\,\ge\, 
1+w(\bar{x}-\bar{k})\, =\, 
1+u (\bar{x})\, >\, c,$$
hence
\begin{equation}\label{7dfuij3jj33ookl}
|\bar{x}-\bar{k}|\ge R-C'
\end{equation}
for a suitable~$C'>0$.

We now observe that
\begin{equation}\label{0044}
\bar{x}-\bar{k}\ge0.\end{equation}
Indeed, if, by contradiction, $\bar{x}-\bar{k}<0$,
we define~$\hat k:=2\bar{x}-\bar{k}<\bar{k}$ and we use~\eqref{7889d12}
to obtain
$$ w(\bar{k}-\bar{x})=w(\bar{x}-\hat k)> u(\bar{x})=w(\bar{x}-\bar{k}).$$
Since~$w$ is even, 
this is a contradiction, and~\eqref{0044} is proved.
\vspace{1mm}

We deduce from~\eqref{909878124556-2},
\eqref{7dfuij3jj33ookl} and~\eqref{0044} that
\begin{equation}\label{6709}
\bar{x}-\bar{k}\in[R-C',\,R].
\end{equation}
We fix~$\kappa\in\R$ such that~$u(-\kappa)=-1+c$.
We remark that~$-\kappa\le \bar{x}$ and so
\begin{equation}\label{2-87dfhjh3nd72h2}
u(x-\kappa) \le u (x+\bar{x}),\end{equation}
for any~$x\in\R$, thanks to~\eqref{87dfhjh3nd72h2}
and the monotonicity of~$u$.
\vspace{1mm}

Now, we
take any
\begin{equation}\label{YY} y\in \left[ \frac{R}{2},\,R\right].
\end{equation}
Then, by~\eqref{6709}, we have that
$$ \bar{x}-y-\bar{k}\in \left[-\frac{R}{2}
,\,\frac{R}{2}\right],$$
and so, by~\eqref{al2},
\begin{eqnarray*}
 1+w(\bar{x}-y-\bar{k}) 
 & \le &  C \big( R+1-|\bar{x}-y-\bar{k}|\big)^{-2s}
\\ & \le &  C (R/2)^{-2s}
\, \le \, 4C\,y^{-2s}.
\end{eqnarray*}
By the above inequality,~\eqref{7889d12}
and~\eqref{2-87dfhjh3nd72h2}
we obtain that
$$ u(-\kappa-y)\, \le\,  u(\bar{x}-y)\, \le\, 
w(\bar{x}-y-\bar{k})\, \le\,  -1+ 4C \,y^{-2s}$$
for any~$y$ as in~\eqref{YY}.
\vspace{1mm}

Since~$\kappa$ is a constant and~$R$ may be taken arbitrarily large, this 
says that, when~$x$ is negative and very large,
$$
u(x)\le -1+C\,|x|^{-2s},
$$
for a suitably renamed $C>0$.
Analogously, one can prove that
$$
u(x)\ge +1-C\,|x|^{-2s}
$$
when~$x$ is positive and very large, and these estimates
prove the formula in~\eqref{EST5}.
\vspace{2mm}

Finally, in order to prove the estimate in~\eqref{stima_der}, we observe that the function $u$ belongs to $C^2(\R)$ (see Lemma~\ref{lem_dif}) and that its derivative 
$u'$ satisfies the following equation
$$
-(-\Delta)^s u'(x) = W''(u(x)) u' \ \ \text{for any} \ x\in\R.
$$
Then, since  $\lim_{x\to\pm\infty} u=\pm 1$  and  the $C^2$ potential $W$ attains its minimum on $\pm 1$, there exist $\alpha, \beta >0$ such that $(u)'$ satisfies
$$
-(-\Delta)^s u'(x) \geq \alpha\,u'(x) \ \ \ \text{for any} \ x\in \R\setminus(-\beta, \beta).
$$
Hence, if we choose $v=u'$, Corollary~\ref{cor_super} yields the desired estimate in~\eqref{stima_der}.\qed
\end{proof}

\vspace{1mm}
\begin{remark}
{\rm We note that the statement in~Proposition~\ref{pro_control} is also valid for solution in $ [0,\infty)$, by replacing the limit condition in~\eqref{eqqq} with the following assumptions
$$
\dys \lim_{x\to+\infty} u(x)=+1 \ \ \ \text{and}\  \ \ u(x)=-1 \ \ \forall x \in (-\infty, 0].
$$
In such a case, the estimates~\eqref{EST5} and~\eqref{stima_der} are meant for $x$ positive and large enough.
}
\end{remark}

\vspace{2mm}

\subsection{A compactness remark}\label{A.B}

In the following lemma, we give full details of a compactness result of classical flavor.

\begin{lemma}\label{C.B}
Let $n\ge1$, $\Omega\subset\R^n$ be a Lipschitz bounded open set 
and $\TT$ be a bounded subset of 
$L^2(\Omega)$.
Suppose that
$$ \sup_{f\in \TT} 
\int_{\Omega}\int_{\Omega}
\frac{|f(x)-f(y)|^2}{|x-y|^{n+2s}}\,dx\,dy
\,<\,+\infty.$$
Then $\TT$ is precompact in $L^2(\Omega)$.
\end{lemma}

\begin{proof} 
The proof 
follows the one of
the classical Riesz-Frechet-Kolmogorov Theorem, but we need to operate some modifications due to the non-locality of 
the fractional norm.

We show that $\TT$ is 
totally 
bounded in 
$L^2(\Omega)$,
i.e., for any $\eps \in (0,1)$ there exist $\beta_1,\dots,\beta_M \in L^2(\Omega)$
such that for any $f\in \TT$ there exists $j\in \{ 1,\dots, M\}$ such
that
\begin{equation}\label{asd}
\| f-\beta_j\|_{L^2(\Omega)} \le\eps.
\end{equation}

First, we remark that we can extend any
function~$f\in\TT$ as a function $\tilde{f}$ in $H^{s}(\R^n)$ (see, for instance,~\cite[Section 5]{DPV11}). Therefore, we can suppose that $\Omega$ is contained in a large cube~$\tilde{\Omega}$, with 
$\|\tilde{f}\|_{H^s(\tilde{\Omega})} \leq C_0 \|{f}\|_{H^s({\Omega})}$. For the sake of simplicity, we drop the tilda's in $f$ and $\Omega$
and we let 
\begin{eqnarray*}&& C:= 1+\sup_{f\in \TT}\|f\|_{L^2(\Omega)}+\sup_{f\in\TT}
\int_{\Omega}\int_{\Omega}
\frac{|f(x)-f(y)|^2}{|x-y|^{n+2s}}\,dx\,dy,\\
&&\rho\leq\rho_{\eps}:= \left(
\frac{\eps}{4\sqrt{C\,n^{(n/2)+1}}}
\right)^{1/s}
\ {\mbox{
and
}}\ \ \
\eta=\eta_{\eps}:=\frac{\eps\,\rho^{n/2}}2,\end{eqnarray*}
and we take a collection of disjoints cubes $Q_1,\dots,Q_N$
of side 
$\rho$
such that
$$
\Omega = \bigcup_{j=1}^N Q_j.
$$
For any $x\in \Omega$ 
we define
\begin{equation}\label{def_jx}
j(x) \ \text{as the unique integer in}  \ \{ 1,\dots,N\} \ \text{for which}
 \ x\in Q_{j(x)}.
\end{equation}
Also, for any $f\in\TT$, let
$$ P(f) (x):=
\frac{1}{|Q_{j(x)}|} 
\int_{Q_{j(x)}}
f(y)\,dy.$$
Notice that
$$
P(f+g)=P(f)+P(g) \ \text{for any} \ f, g \in \TT
$$
and that $P(f)$ is constant, say equal to $q_j(f)$, in any $Q_j$,
for $j\in\{ 1,\dots,N\}$. Therefore, we can define
$$ R(f) :=\rho^{n/2} \big( q_1(f),\dots,q_N(f)\big)\in \R^N.$$
We observe that $R({f+g})=R(f)+R(g)$. Moreover, 
\begin{eqnarray}\label{509}
\|P(f)\|^2_{L^2(\Omega)}  &  = & \sum_{j=1}^N\int_{Q_j}|P(f)|^2\, dx \nonumber \\
\nonumber \\
& \leq &  \rho^n\sum_{j=1}^N |q_j(f)|^2 \ = \ |R(f)|^2 \ \leq \ \frac{|R(f)|^2}{\rho^n}.
\end{eqnarray}
and, by H\"older
inequality,
\begin{eqnarray*}
|R(f)|^2  & = &  \sum_{j=1}^N\rho^n|q_j(f)|^2
\, = \, \frac{1}{\rho^n}
\sum_{j=1}^N \left|
\int_{Q_{j}} f(y)\,dy
\right|^2  \\
& \leq & 
\sum_{j=1}^N 
\int_{Q_{j}} |f(y)|^2\,dy
\, = \,
\int_{\Omega} |f(y)|^2 \, =\, \|f\|_{L^2(\Omega)}^2.
\end{eqnarray*}
In particular,
$$
\sup_{f\in\TT}|R(f)|^2
\le C,
$$
that is, the set $R(\TT)$ is bounded in $\R^N$
and so, since it is finite dimensional, it is totally bounded.
Therefore, there exist $b_1,\dots,b_M\in \R^N$ such that
\begin{equation}\label{eq_22bis}
R(\TT)\subseteq \bigcup_{i=1}^M B_\eta (b_i).
\end{equation}
For any $i\in \{1,\dots, M\}$, we write the coordinates of $b_i$ as $b_i=(b_{i,1},\dots, b_{i,N})\in \R^N$.
For any $x\in \Omega$, we set
$$
\beta_i(x):= \rho^{-n/2}\,b_{i,j(x)},
$$
where $j(x)$ is as in~\eqref{def_jx}.

Notice that $\beta_i$ is constant on $Q_j$, i.e. if $x\in Q_j$ then 
\begin{equation}\label{eq_2t2}
P(\beta_i)(x)=\rho^{-\frac{n}{2}}b_{i,j}=\beta_{i}(x)
\end{equation}
and so $q_j(\beta_i)=\rho^{-\frac{n}{2}}b_{i,j}$; thus\
\begin{equation}\label{eq_25bis}
 R(\beta_i)=b_i.
 \end{equation}

Furthermore, for any $f\in \TT$, by H\"older inequality,
\begin{eqnarray*}
\| f-P(f)\|_{L^2(\Omega)}^2  & = & \sum_{j=1}^N
\int_{Q_j} |f(x)-P(f)(x)|^2\,dx
\\
\\ & = &   \sum_{j=1}^N\int_{Q_j} \left| f(x)-\frac{1}{|Q_j|} \int_{Q_j}
f(y)\,dy\right|^2\,dx
\\
\\ & = &  \sum_{j=1}^N
\int_{Q_j} \frac{1}{|Q_j|^2}\left|\int_{Q_j} f(x)-
f(y)\,dy\right|^2\,dx
\\
& \le &  \frac{1}{\rho^n}\sum_{j=1}^N
\int_{Q_j} \left[\int_{Q_j} \big|f(x)-
f(y)\big|^2\,dy\right]\,dx
\\
\\ & \le  &  n^{(n/2)+1}\rho^{2s}\sum_{j=1}^N
\int_{Q_j} \left[\int_{Q_j} \frac{|f(x)-
f(y)|^2}{|x-y|^{n+2s}}\,dy\right]\,dx
\\
\\ & \le &  n^{(n/2)+1}\rho^{2s}\sum_{j=1}^N
\int_{Q_j} \left[\int_{\Omega} \frac{|f(x)-
f(y)|^2}{|x-y|^{n+2s}}\,dy\right]\,dx
\\
\\ &= &  n^{(n/2)+1}\rho^{2s}
\int_{\Omega} \left[\int_{\Omega} \frac{|f(x)-
f(y)|^2}{|x-y|^{n+2s}}\,dy\right]\,dx
\\
\\ & \le &  C\,n^{(n/2)+1}\,\rho^{2s} \, = \, \frac{\eps^2}{16}. 
\end{eqnarray*}
Consequently, for any $j\in\{1, ..., M\}$, recalling~\eqref{509} and~\eqref{eq_2t2}
\begin{eqnarray}\label{905}
\|f-\beta_j\|_{L^2(\Omega)} 
& \le & 
\|f-P(f)\|_{L^2(\Omega)}+\|P(\beta_j)-\beta_j\|_{L^2(\Omega)}
+\|P(f-\beta_j)\|_{L^2(\Omega)} \nonumber \\
& \leq &   \frac{\eps}{2}+\frac{|R(f)-R(\beta_j)|}{\rho^{n/2}}.
\end{eqnarray}
\vspace{1mm}

Now, given any~$f\in\TT$, we recall~\eqref{eq_22bis} and~\eqref{eq_25bis} and we take~$j\in\{1,\dots,M\}$ such that~$R(f)\in
B_\eta(b_j)$.
Then,~\eqref{eq_2t2} and~\eqref{905} give that
\begin{eqnarray*}
\| f-\beta_j\|_{L^2(\Omega)} \ \le
\ \frac{\eps}{2}+\frac{|R(f)-b_j|}{\rho^{n/2}}\ \le \
\frac\eps2+\frac{\eta}{\rho^{n/2}}\ = \ \eps.
\end{eqnarray*}
This proves \eqref{asd}, as desired.\qed
\end{proof}

\vspace{2mm}

\subsection{Integral computations}

\vspace{2mm}

Lemma~\ref{lem_inout} deals with the kernels of the Gagliardo norm in the case of $n$-dimensional balls $B_R$. We provide a lower bound, with respect to the radius $R$ of the contribution coming from far of the energy.
\vspace{1mm}

Lemma \ref{lip_lem-0} and Lemma \ref{lip_lem-1} estimate the fractional derivative of bounded functions on the whole space $\R^n$. We also provide some estimates of the energy with respect to the $L^{\infty}$-norm of the functions and their derivatives. The case of radial symmetric functions is analyzed in Lemma \ref{lip_lem-2}.
\vspace{1mm}

\begin{lemma}\label{lem_inout}
Let~$n\ge 1$ and~$R\ge1$. Then,
\begin{equation}\label{div}
\text{if} \ s \in (0,\, 1/2), \ \ \ \int_{B_R}\int_{B_{2R}\setminus B_R}
\frac{dx\,dy}{|x-y|^{n+2s}}
\le \frac{3\,\omega_{n-1}^2\,R^{n-2s}}{2s\,(1-2s)}.
\end{equation}
\begin{equation}\label{div2}
\text{If} \ s=1/2, \ \ \ 
\int_{B_R}\int_{\CC B_{R+1}}
\frac{dx\,dy}{|x-y|^{n+2s}}\, \leq \, \omega_{n-1}^2\,R^{n-1}\,\left( 2^n+\log(3R)\right).
\end{equation}
\begin{equation}\label{div3}
\text{If} \ s \in (1/2,\,1), \ \ \
\int_{B_R}\int_{\CC B_{R+1}}
\frac{dx\,dy}{|x-y|^{n+2s}}\, \leq \, 
\frac{\omega_{n-1}^2\, R^{n-1}}{2s-1}.
\end{equation}
\end{lemma}

\begin{proof} For any fixed~$y\in\R^n$,
\begin{eqnarray}\label{vidi}
2s\!\int_{B_1}\frac{dx}{|x-y|^{n+2s}}
 & = & -\int_{B_1}\,{\rm div}\,\left(\frac{x-y}{|x-y|^{n+2s}}
\right)\,dx \nonumber
\\ \nonumber
\\ & = &  -\int_{\partial B_1}
\frac{x-y}{|x-y|^{n+2s}}\cdot x
\,d{\mathcal{H}}^{n-1}(x) \nonumber
\\ \nonumber
\\ &\le & 
\int_{\partial B_1}
|x-y|^{1-n-2s}
\,d{\mathcal{H}}^{n-1}(x).\end{eqnarray}
Accordingly, if~$s\in(0,\,1/2)$,
\begin{eqnarray*}
2s\!\int_{B_1}\!\int_{B_2\setminus B_1}
\frac{dx\,dy}{|x-y|^{n+2s}}  & \le & 
\int_{\partial B_1}\left[\int_{B_2\setminus B_1}
|x-y|^{1-n-2s}\,dy\right]
\,d{\mathcal{H}}^{n-1}(x)
\\
\\ & \le &  \int_{\partial B_1}\left[\int_{B_3}
|\zeta|^{1-n-2s}\,d\zeta\right]
\,d{\mathcal{H}}^{n-1}(x) \\
\\
&= &  \frac{3^{1-2s}\,\omega_{n-1}^2}{1-2s},
\end{eqnarray*}
which is finite by our assumption on~$s$, and so,
by changing
variable~$\tilde x:=x/R$ and~$\tilde y:=y/R$,
\begin{eqnarray*}
2s\!\int_{B_R}\int_{B_{2R}\setminus B_R}
\frac{dx\,dy}{|x-y|^{n+2s}}  & = &  R^{n-2s}\,
\int_{B_1}\int_{B_{2}\setminus B_1}
\frac{d\tilde x\,d\tilde y}{|\tilde x-\tilde y|^{n+2s}}
\\
\\& \le & 
\frac{3^{1-2s}\,\omega_{n-1}^2\,R^{n-2s}}{1-2s},
\end{eqnarray*}
proving~\eqref{div}.

\vspace{2mm}

On the other hand, if~$s \in (1/2,\,1)$, we set~$\eps:=1/R$,
and we use~\eqref{vidi}
to conclude that
\begin{eqnarray*}
 \int_{B_1}\!\int_{\CC B_{1+\eps}}
\frac{dx\,dy}{|x-y|^{n+2s}}
 & \le &  \int_{\partial B_1}\left[
\int_{\CC B_{1+\eps}} 
|x-y|^{1-n-2s}\,dy\right]
d{\mathcal{H}}^{n-1}(x)
\\
\\ & \le & 
\int_{\partial B_1}\left[
\int_{\CC B_{\eps}} 
|\zeta|^{1-n-2s}\,d\zeta\right]
d{\mathcal{H}}^{n-1}(x)
\, \le\, \frac{\omega_{n-1}^2\varepsilon^{1-2s}}{2s-1},
\end{eqnarray*}
hence~\eqref{div3} follows from scaling.

\vspace{2mm}

Finally, when~$s=1/2$, we use~\eqref{vidi} in the following
way:
\begin{eqnarray*}
\int_{B_1}\int_{\CC B_{1+\eps}}
\frac{dx\,dy}{|x-y|^{n+2s}}
 & \le & 
\int_{B_1}\int_{B_2\setminus B_{1+\eps}}
\frac{dx\,dy}{|x-y|^{n+1}}
+
\int_{B_1}\int_{\CC B_{2}}
\frac{dx\,dy}{(|y|/2)^{n+1}}
\\ 
\\ & \le &  \int_{\partial B_1}\left[
\int_{B_2\setminus B_{1+\eps}} 
|x-y|^{-n}\,dy\right]
d{\mathcal{H}}^{n-1}(x)+2^{n}\omega_{n-1}^2
\\ 
\\&\le & 
\int_{\partial B_1}\left[
\int_{B_3\setminus B_{\eps}} 
|\zeta|^{-n}\,d\zeta\right]
d{\mathcal{H}}^{n-1}(x)+2^{n}\omega_{n-1}^2
\\
\\ & = &  \omega_{n-1}^2\left( 2^n+\log\frac3{\eps}\right),
\end{eqnarray*}
hence~\eqref{div2} follows again from scaling.\qed
\end{proof}
\vspace{1mm}

Similarly as in previous Lemma~\ref{lem_inout}, one can estimate the kernel interaction of smooth functions as follows. 

\begin{lemma}\label{lip_lem-0}
Let~$n\ge1$ and 
$x\in\R^n$, $\rho>0$ and~$\psi\in L^\infty(\R^n)\cap
W^{1,\infty}(B_\rho(x))$. Then, \vspace{.5mm}
\begin{equation}\label{909-0}\int_{\R^n}
\frac{
|\psi(x)-\psi(y)|^2
}{
|x-y|^{n+2s}}\,dy\le
\frac{4\,\omega_{n-1}}{(1-s)\,s}\,\Big[
\|\nabla\psi\|^2_{L^\infty (B_\rho(x))} \rho^{2(1-s)}+
\|\psi\|^2_{L^\infty (\R^n)}\rho^{-2s}
\Big].\end{equation}
\end{lemma}

\begin{proof} We bound the left hand side of \eqref{909-0}
by
\begin{eqnarray*} 
  & & \int_{ B_\rho(x)}
\frac{
|\psi(x)-\psi(y)|^2
}{
|x-y|^{n+2s}}\,dy \! + \!
\int_{\CC B_\rho(x)}
\frac{
|\psi(x)-\psi(y)|^2
}{
|x-y|^{n+2s}}\,dy
\\
\\  &  & \qquad \qquad \qquad \qquad \le
\int_{ B_\rho(x)}
\frac{ \|\nabla \psi\|^2_{L^\infty(B_\rho(x))}
}{
|x-y|^{n+2s-2}}\,dy +
\int_{\CC B_\rho(x)}
\frac{4\,\|\psi\|_{L^\infty(\R^n)}^2
}{
|x-y|^{n+2s}}\,dy
\\
\\ && \qquad \qquad \qquad \qquad =
\int_{ B_\rho} 
\frac{ \|\nabla \psi\|^2_{L^\infty(B_\rho(x))}
}{
|\zeta|^{n+2s-2}}\,d\zeta +
\int_{\CC B_\rho}
\frac{4\,\|\psi\|_{L^\infty(\R^n)}^2
}{
|\zeta|^{n+2s}}\,d\zeta
\\
\\ & & \qquad \qquad \qquad \qquad \leq \omega_{n-1} \left(\frac{ \|\nabla \psi\|^2_{L^\infty(B_\rho(x))}
\,\rho^{2(1-s)}}{2(1-s)}+\frac{4 
\,\|\psi\|^2_{L^\infty(\R^n)}\rho^{-2s}}{s}
\right)
\end{eqnarray*}
and this easily implies \eqref{909-0}.\qed
\end{proof}

\vspace{1mm}

\begin{lemma}\label{lip_lem-1}
Let~$n\ge1$.  
Let~$x\in\R^n$, $\rho>0$ and~$\psi\in 
L^\infty(\R^n)$. 

Suppose that there exists~$\Xi \in\R^n$ and~$K\in\R$
\begin{equation}\label{HY1}
\psi(y)-\psi(x)-\Xi \cdot(y-x)\le K |x-y|^2,
\end{equation}
for any~$y\in B_\rho(x)$. Then,
\begin{equation}\label{HTH1}
\int_{\R^n}
\frac{
\psi(y)-\psi(x)
}{
|x-y|^{n+2s}}\,dy \le
\omega_{n-1}\,\Big(\frac{
K \rho^{2(1-s)}}{2(1-s)}+\frac{
\|\psi\|_{L^\infty (\R^n)}\rho^{-2s}}{s}
\Big).\end{equation}
Analogously, if we replace~\eqref{HY1} with
the assumption that there exists~$\tilde\Xi\in\R^n$
and~$\tilde{K}\in\R$ such 
that
\begin{equation}\label{HY2}
\psi(y)-\psi(x)-\tilde\Xi \cdot(y-x)\ge -\tilde{K} |x-y|^2,
\end{equation}
for any~$y\in B_\rho(x)$, we obtain that
\begin{equation}\label{HTH2}
\int_{\R^n}
\frac{
\psi(x)-\psi(y)
}{
|x-y|^{n+2s}}\,dy \le {\omega_{n-1}}\, \Big(
\frac{\tilde{K} \rho^{2(1-s)}}{2(1-s)}+\frac{
\|\psi\|_{L^\infty (\R^n)}\rho^{-2s}}{s}
\Big).\end{equation}
In particular, if~$\psi\in L^\infty(\R^n)\cap
W^{2,\infty}(B_\rho(x))$ we have that
\begin{eqnarray}\label{909-1}
&&\left|\int_{\R^n}
\frac{
\psi(y)-\psi(x)
}{
|x-y|^{n+2s}}\,dy\right| \nonumber\\
&&\quad\quad \quad \le
\frac{\omega_{n-1}}{(1-s)\,s}\Big(
\|D^2\psi\|_{L^\infty (B_\rho(x))} \rho^{2(1-s)}+
\|\psi\|_{L^\infty (\R^n)}\rho^{-2s}
\Big).
\end{eqnarray}
\end{lemma}

\begin{proof} 
We prove~\eqref{HTH1}
under assumption~\eqref{HY1}, since the proof of~\eqref{HTH2}
under assumption~\eqref{HY2} is the same, and then~\eqref{909-1}
follows from~\eqref{HY1} and~\eqref{HY2}
by choosing~$\Xi=\tilde\Xi=\nabla\psi(x)$ 
and~$K=\tilde{K}:=\|D^2\psi\|_{L^\infty (B_\rho(x))}$.
The proof below is similar to the one of
Lemma~\ref{lip_lem-0}, but we give the details for the
facility of the reader.
\vspace{1mm}

Notice that, by symmetry,
\begin{equation*} \int_{ B_\rho(x)}
\frac{
\Xi\cdot(x-y)
}{
|x-y|^{n+2s}}\,dy =0.\end{equation*}

Consequently,
we bound the left hand side of \eqref{HTH1}
by
\begin{eqnarray*} 
&& \int_{ B_\rho(x)}
\frac{
\psi(y)-\psi(x)+\Xi \cdot(x-y)
}{
|x-y|^{n+2s}}\,dy+
\int_{\CC B_\rho(x)}
\frac{
|\psi(x)-\psi(y)|
}{
|x-y|^{n+2s}}\,dy
\\
\\ && \qquad \qquad \qquad  \qquad \ \ \le
\int_{ B_\rho(x)}
\frac{ K
}{
|x-y|^{n+2s-2}}\,dy +
\int_{\CC B_\rho(x)}
\frac{2\|\psi\|_{L^\infty(\R^n)}
}{
|x-y|^{n+2s}}\,dy
\\
\\ && \qquad \qquad \qquad \qquad \ \  =
\int_{ B_\rho} 
\frac{ K
}{
|\zeta|^{n+2s-2}}\,d\zeta +
\int_{\CC B_\rho}
\frac{2\|\psi\|_{L^\infty(\R^n)}
}{
|\zeta|^{n+2s}}\,dy
\\
\\ && \qquad \qquad \qquad \qquad  \ \ = \omega_{n-1} \Big[\frac{K
\,\rho^{2(1-s)}}{2(1-s)}+\frac{\|\psi\|_{L^\infty(\R^n)}\rho^{-2s}}{s}
\Big],
\end{eqnarray*}
that is~\eqref{HTH1}.\qed
\end{proof}

\vspace{1mm}

\begin{lemma}\label{lip_lem-2} 
Let~$n\ge1$ 
and let~$x\in\R^n$. Let~$\psi\in
L^\infty(\R^n)$ be continuous, radial and radially non-decreasing, with
$$
 \sup_{\R^n}\psi=\max_{\R^n}\psi=M.
$$
Suppose that $\psi\in W^{2,\infty}(\{\psi<M\})$.
Then,
\begin{equation}\label{s778}
\int_{\R^n}
\frac{
\psi(x)-\psi(y)
}{
|x-y|^{n+2s}}\,dy\le
\frac{\omega_{n-1}}{(1-s)\,s} \Big(
\|D^2 \psi\|_{L^\infty(
\{\psi<M\})} 
+\|\psi\|_{L^\infty(\R^n)}
\Big).\end{equation}
\end{lemma}

\begin{proof} 
By the radial symmetry of~$\psi$, we have that
$$ \{\psi<M
\} = B_\kappa$$
for some~$\kappa>0$.
Accordingly,
\begin{equation}\label{d77djkjkw-2}\begin{split}&{\mbox{ for any $z$, $y$ 
in the closure of $B_\kappa$,}}\\&\qquad
\psi(y)\ge \psi(z)+\nabla\psi(z)(y-z)-\|D^2 \psi\|_{L^\infty(
\{\psi<M
\})} 
(z-y)^2.
\end{split}\end{equation}\vspace{1mm}

Also, fixed any~$x\in\R^n$, we define
$$ z:=\left\{
\begin{matrix}
x & {\mbox{ if $x\in B_\kappa$,}}\\
\kappa x/|x| &{\mbox{ otherwise.}}\end{matrix}
\right.$$
Notice that~$|z|\leq \kappa$, that~$\psi (x)=\psi(z)$, that~$\psi(z)-\psi(y)\ge 0$
if and only if~$|z|\ge |y|$. Also, if $|x|>\kappa$ and~$\alpha$ is the angle between
the vector~$x-z$ and~$y-z$, the convexity of~$B_\kappa$
implies that
$$ \alpha\,\in\,\left[ \frac{\pi}2,\frac{3\pi}{2}\right]$$
and so~$\cos\alpha\le0$.
Hence,
\begin{eqnarray*} |x-y| &= & \sqrt{ |z-y|^2+|z-x|^2 -2|z-y|\,|z-x|\cos\alpha
} \\
&\ge &  \sqrt{ |z-y|^2+|z-x|^2 } \ \ge  \ |z-y|,
\end{eqnarray*}
if $|x|>\kappa$ (and, obviously, the estimate holds for $|x|\leq \kappa$ too, since is the case $z=x$).

Thus, we use
the above observations
to obtain
\begin{eqnarray*}
 \int_{\R^n} \frac{ 
\psi(x)-\psi(y)}{|x-y|^{n+2s}}\,dy & \le & \int_{B_{|z|}}
\frac{
\psi(z)-\psi(y)
}{
|x-y|^{n+2s}}\,dy
\ \le \ \int_{B_{|z|}}
\frac{
\psi(z)-\psi(y)
}{
|z-y|^{n+2s}}\,dy
\\
\\ \qquad &\le & 
\int_{B_{|z|}\cap B_1(z)}
\frac{
\psi(z)-\psi(y)
}{
|z-y|^{n+2s}}\,dy
\\
&& +\int_{B_{|z|}\cap \CC B_1(z)}
\frac{
\psi(z)-\psi(y)
}{
|z-y|^{n+2s}}\,dy
\\
\\  \qquad &\le &
\int_{B_1(z)}
\|D^2 \psi\|_{L^\infty(
\{\psi<M\})}
\,|z-y|^{2-n-2s}\,dy
\\ &&+\int_{\CC B_1(z)}
\frac{ 2\,\|\psi\|_{L^\infty(\R^n)}
}{
|z-y|^{n+2s}}\,dy
\\
\\ \qquad &\le &
{\omega_{n-1}}\left[
\frac{ \|D^2 \psi\|_{L^\infty(
\{\psi<M\})} }{1-s}+\frac{\|\psi\|_{L^\infty(\R^n)} }{s}
\right] ,
\end{eqnarray*}
which implies the desired result.\qed
\end{proof}

\vspace{2mm}

\noindent
{\bf Acknowledgments}. The authors would like to thank Luis Silvestre for his useful comments.

\vspace{2mm}


\vspace{2mm}


\begin{thebibliography}{28}
\addcontentsline{toc}{section}{References}

\bibitem{alberti2000}{Alberti, G.}: {Some remarks about a notion of rearrangement}. {Ann. Scuola Norm. Sup. Pisa Cl. Sci.} {\bf 4}, 457--472 (2000)
  
\bibitem{alberti2001} {Alberti, G., Ambrosio, L., Cabr\'e, X.}: {On a long-standing conjecture of E.~De~Giorgi: symmetry in 3D for general nonlinearities and a local minimality property}. {Acta Appl. Math.} {\bf 65}, no. 1-3, 9--33 (2001)

\bibitem{alberti} {Alberti, G., Bellettini, G.}: {A nonlocal anisotropic model for phase transitions I: the optimal profile problem}.{Math. Ann.} {\bf 310},  527--560 (1998)

\bibitem{AL89} {Almgren, F.~J., Lieb, E.~H.}: Symmetric decreasing rearrangement is sometimes continuous. {J. Amer. Math. Soc.} {\bf 2}, no.~4, 683--773~(1989)

\bibitem{ambrosio} {Ambrosio, L, Cabr\'e, X.}: {Entire solutions of semilinear elliptic equations in $\R^3$ and a conjecture of De~Giorgi}, {J. Amer. Math. Soc.} {\bf 13}, no. 4, 725--739~(2000)

\bibitem{ambrosio2010} {Ambrosio, L., de Philippis, G.,  Martinazzi, L.}: {Gamma-convergence of nonlocal perimeter functionals}, {Manuscripta Math.} {\bf 134}, no. 3-4, 377--403~(2011)

\bibitem{Ba94} {Baernstein, A.}: {A unified approach to symmetrizations}. In: Alvino, A. et al. (eds.) {Partial differential equations of elliptic type}, pp. 47--91. Symposia Math. 35, Cambridge Univ. Press (1994)

\bibitem{cabre2010} {Cabr\'e, X., Cinti, E.}: {Energy estimates and 1-D symmetry for nonlinear equations involving the half-Laplacian}. {Discrete and Continuous Dynamical Systems} {\bf 28}, 1179-1206~(2010)

\bibitem{cabre} {Cabr\'e, X., Sol\`a-Morales, J.}: Layer solutions in a half-space for boundary reactions. {Commun. Pure Appl. Math.} {\bf 5}, no. 12, 1678--1732~(2005)

\bibitem{cabre2011} {Cabr\'e, X., Sire, Y.}: {Nonlinear equations for fractional Laplacians I: Regularity, maximum principles, and Hamiltonian estimates}. Preprint. \url{http://arxiv.org/abs/1012.0867}~(2010) 

\bibitem{cabre21th} {Cabr\'e, X., Sire, Y.}: Nonlinear equations for fractional Laplacians II: existence, uniqueness and qualitative properties of solutions. {In preparation}.

\bibitem{CS10} {Caffarelli, L., Silvestre, L.}: Regularity results for nonlocal equations by approximation.
To appear in {Arch. Rational Mech. Anal.}. \url{http://arxiv.org/abs/0902.4030v2}~(2010) 

\bibitem{caffarelli} {Caffarelli, L.~\!A., Valdinoci, E.}: {Uniform estimates and limiting arguments for nonlocal minimal surfaces}. {Calc. Var. Partial Differential Equations} {\bf 41}, no. 1-2, 203--240~(2011)

\bibitem{carothers} {Carothers, N.~\!L.}: {Real Analysis}. Cambridge University Press, Cambridge~(1999)

\bibitem{DPV11} {Di~Nezza, E., Palatucci, G., Valdinoci, E.}: {Hitchhiker's guide to the fractional Sobolev spaces}. {Submitted paper}.
\url{http://arxiv.org/abs/1104.4345v2}~(2011)

\bibitem{FV11} {Farina, A., Valdinoci, E.}: {Rigidity results for elliptic PDEs with uniform limits~: an abstract framework with applications}. To appear in~{Indiana Univ. Math. J.} \url{http://www.iumj.indiana.edu/IUMJ/Preprints/4433.pdf}~(2011)

\bibitem{garroni} {Garroni, A., Palatucci, G.}: {A singular perturbation result with a fractional norm}. In: {Variational problems in material 
science}, {Progress in Nonlinear Differential Equations and 
Their Applications} {\bf 68}, pp. 111--126.
Birkh\"auser, Basel~(2006)


\bibitem{garsia}  {Garsia, A.~\!M., Rodemich, E.}: {
Monotonicity of certain functionals under rearrangement}. {Ann. Inst. Fourier (Grenoble)} {\bf 24}, no. 2, vi, 67--116~(1974)

\bibitem{gonzalez}  {Gonz{\'a}lez, M.~\!d.~\!M.}: {Gamma convergence of an energy functional related to the fractional 
Laplacian}. {Calc. Var. Partial Differential Equations} {\bf 36}, no. 2, 
173--210~(2009)

\bibitem{sv2011}  {Savin, O., Valdinoci, E.}:
{Density estimates for a nonlocal variational model via the Sobolev inequality}. {Submitted paper}. \url{http://arxiv.org/abs/1103.6205}~(2011)


\bibitem{sv0}  {Savin, O., Valdinoci, E.}:
{Density estimates for a variational model driven by the
Gagliardo norm}. {Submitted paper}. \url{http://arxiv.org/abs/1007.2114v3}~(2011)

\bibitem{sv1}  {Savin, O., Valdinoci, E.}:
{$\Gamma$-convergence for nonlocal phase transitions}. {Submitted paper}. \url{http://arxiv.org/abs/1007.1725v3}~(2011)


\bibitem{silvestre}  {Silvestre, L.}: {Regularity of the obstacle problem for a fractional power of the Laplace operator}. {Ph.D. Thesis}, Austin University.
\url{http://www.math.uchicago.edu/~luis/preprints/luisdissreadable.pdf}~(2005)

\bibitem{silvestre07}  {Silvestre, L.}: Regularity of the obstacle problem for a fractional power of the Laplace operator.
{Communications on Pure and Applied Mathematics} {\bf 60}, no. 1, 67--112~(2007)

\bibitem{SV09} {Sire, Y., Valdinoci, E.}: Fractional Laplacian phase transitions and boundary reactions: a 
geometric inequality and a symmetry result. {J. Funct. Anal.} {\bf 256}, no. 6, 1842--1864~(2009)

\bibitem{SV09b} {Sire, Y., Valdinoci, E.}: {Rigidity results for some boundary quasilinear phase transitions}. {Comm. Partial Differential 
Equations} {\bf 33}, no. 7-9, 765--784~(2009)

\bibitem{ssv}  {Valdinoci, E., Sciunzi, B., Savin, V.\!~O.}:
{Flat level set 
regularity of $p$-Laplace phase transitions}. {Mem. Amer. Math. Soc.} {\bf 182}, no. 858, vi+144 pp.~(2006)

\bibitem{wheeden}  {Wheedenden, R.\!~L., Zygmund, A.}:
{Measure and integral: 
An introduction to real analysis}. {Pure and Applied Mathematics} {\bf 43}, 
Marcel Dekker, Inc., New York-Basel, x+274~pp.~(1977)

\end{thebibliography}
\end{document}